\documentclass[final,10pt,a4paper]{article}
\usepackage[T1]{fontenc}

\usepackage{amsmath}
\usepackage{amsfonts}
\usepackage{amssymb}
\usepackage{amsthm}
\usepackage{enumerate}
\usepackage{url}

\usepackage[margin=1.5in]{geometry}

\usepackage{hyperref}\hypersetup{
    colorlinks=true,
    linkcolor=blue,
    filecolor=magenta,      
    urlcolor=cyan,
    bookmarks=true,
    pdfpagemode=FullScreen,
}

\usepackage[backgroundcolor=white]{todonotes}

\newtheorem{theorem}{Theorem}[section]
\newtheorem{lemma}[theorem]{Lemma}
\newtheorem{proposition}[theorem]{Proposition}
\newtheorem{df}[theorem]{Definition}
\newtheorem{kor}[theorem]{Corollary}
\newtheorem{conj}[theorem]{Conjecture}
\newtheorem*{rem}{Remark}

\usepackage{authblk}

\title{Extension of a Diophantine triple with the property $D(4)$}
\author[a]{Marija Bliznac Trebje\v{s}anin}
\affil[a]{University of Split, Faculty of Science\\ Ru\dj{}era Bo\v{s}kovi\'{c}a 33, 21000 Split, Croatia\\
Email: marbli@pmfst.hr}
\date{}

\begin{document}
\maketitle

\begin{abstract} In this paper, we give an upper bound on the number of extensions of a triple to a quadruple for the Diophantine $m$-tuples with the property $D(4)$. We also confirm the conjecture of the uniqueness of such an extension in some special cases.
\end{abstract}

\noindent 2010 {\it Mathematics Subject Classification:} 11D09, 11D45, 11J86
\\ \noindent Keywords: diophantine tuples, Pell equations, reduction method, linear forms in logarithms.

\section{Introduction}
\begin{df} Let $n\neq0$ be an integer. We call a set of $m$ distinct positive integers a $D(n)$-$m$-tuple, or $m$-tuple with the property $D(4)$, if the product of any two of its distinct elements increased by $n$ is a perfect square.
\end{df}

One of the most interesting and frequently studied questions is how large these sets can be. In the classical case (when $n=1$), first studied by Diophantus, Dujella proved in \cite{duje_kon} that a $D(1)$-sextuple does not exist and that there are at most finitely many quintuples. Over the years many authors have improved the upper bound for the number of $D(1)$-quintuples. Finally, in \cite{petorke}, He, Togb\'{e} and Ziegler established the nonexistence of $D(1)$-quintuples. Details of the history of the problem with all references are available at this Web address \cite{duje_web}. 

Variants of the problem when $n=4$ or $n=-1$ are also studied frequently. In the case $n=4$, similar conjectures and observations can be made as in the $D(1)$ case. 
In the light of this observation, Filipin and the author have proven in \cite{nas2} that a $D(4)$-quintuple also does not exist. 

In both cases, $n=1$ and $n=4$, conjectures about the uniqueness of an extension of a triple to a quadruple with a larger element are still open. Moreover, in the case $n=-1$, a conjecture about the nonexistence of a quadruple is studied. For a survey of the latter case of the problem one can see \cite{cipu}.

A $D(4)$-pair can be extended with a larger element $c$ to form a $D(4)$-triple. The smallest such $c$ is $c=a+b+2r$, where $r=\sqrt{ab+4}$ and such a triple is often called a regular triple, or in the $D(1)$ case it is also called an Euler triple. There are infinitely many extensions of a pair to a triple and they can be studied by finding solutions to a Pellian equation
\begin{equation}\label{par_trojka}
bs^2-at^2=4(b-a),\end{equation}
where $s$ and $t$ are positive integers defined by $ac+4=s^2$ and $bc+4=t^2.$

For a $D(4)$-triple $\{a,b,c\}$, $a<b<c$, we define
$$d_{\pm}=d_{\pm}(a,b,c)=a+b+c+\frac{1}{2}\left(abc\pm \sqrt{(ab+4)(ac+4)(bc+4)}\right).$$
It is straightforward to check that $\{a,b,c,d_{+}\}$  is a $D(4)$-quadruple, which we will call a regular quadruple. If $d_{-}\neq 0$ then $\{a,b,c,d_{-}\}$ is also a regular $D(4)$-quadruple with $d_{-}<c$.

\begin{conj}
Any $D(4)$-quadruple is regular.
\end{conj}
 
Results which support this conjecture in some special cases can be found for example in \cite{bf}, \cite{dujram}, \cite{fil_par} and \cite{fht}. Some of these results are stated in the next section and will be used in our proofs.

In \cite{fm}, Fujita and Miyazaki approached this conjecture in the $D(1)$ case differently. They examined how many possibilities there are to extend a fixed Diophantine triple with a larger integer. They improved their result from \cite{fm} further in their joint work with Cipu \cite{cfm}, where they have shown that any triple can be extended to a quadruple in at most $8$ ways.

In this paper, we will follow the approach and ideas from \cite{cfm} and \cite{fm} in order to deduce similar results for extensions of a $D(4)$-triple. Usually, the numerical bounds and coefficients are slightly better in the $D(1)$ case, which can be seen after comparing for instance Theorem \ref{teorem1.5} and \cite[Theorem 1.5]{fm}. In order to overcome this problem, we deduce a better numerical lower bound on the element $b$ in an irregular $D(4)$-quadruple, and consider separately all special cases which appeared in the proof (similar to \cite{nas2}).

\medskip
Let $\{a,b,c\}$ be a $D(4)$-triple which can be extended to a quadruple with an element $d$. Then there exist positive integers $x,y,z$ such that
$$ad+4=x^2,\quad bd+4=y^2, \quad cd+4=z^2.$$
By eliminating $d$ from these equations we get a system of generalized Pellian equations
\begin{align}
cx^2-az^2&=4(c-a),\label{prva_pelova_s_a}\\
cy^2-bz^2&=4(c-b).\label{druga_pelova_s_b}
\end{align}

There exists only finitely many fundamental solutions $(z_0,x_0)$ and $(z_1,y_1)$ to these Pellian equations and any solution to the system satisfy $z=v_m=w_n$, where $m$ and $n$ are non-negative integers and  ${v_m}$ and ${w_n}$ are recurrent sequences defined by
\begin{align*}
&v_0=z_0,\ v_1=\frac{1}{2}\left(sz_0+cx_0\right),\ v_{m+2}=sv_{m+1}-v_{m},\\
&w_0=z_1,\ w_1=\frac{1}{2}\left(tz_1+cy_1 \right),\ w_{n+2}=tw_{n+1}-w_n.
\end{align*}

The initial terms of these sequences  were determined by Filipin in \cite[Lemma 9]{fil_xy4}. Notably, one of the results of this paper is an improvement of that lemma by eliminating the case where $m$ and $n$ are even and $|z_0|$ is not explicitly determined.

\begin{theorem}\label{teorem_izbacen_slucaj}
Suppose that $\{a,b,c,d\}$ is a $D(4)$-quadruple with $a<b<c<d$ and that  $w_m$ and  $v_n$  are defined as before. 
\begin{enumerate}
\item[i)] If  equation $v_{2m}=w_{2n}$ has a solution, then $z_0=z_1$ and $|z_0|=2$ or $|z_0|=\frac{1}{2}(cr-st)$.
\item[ii)] If equation $v_{2m+1}=w_{2n}$ has a solution, then $|z_0|=t$, $|z_1|=\frac{1}{2}(cr-st)$ and $z_0z_1<0$.
\item[iii)] If equation $v_{2m}=w_{2n+1}$ has a solution, then $|z_1|=s$, $|z_0|=\frac{1}{2}(cr-st)$ and $z_0z_1<0$.
\item[iv)] If equation $v_{2m+1}=w_{2n+1}$ has a solution, then $|z_0|=t$, $|z_1|=s$ and $z_0z_1>0$.
\end{enumerate}
Moreover, if $d>d_+$, case $ii)$ cannot occur.
\end{theorem}

Also, we improved a bound on $c$ in the terms of $b$ for which an irregular extension might exist.

\begin{theorem}\label{teorem1.5}
Let $\{a,b,c,d\}$ be a $D(4)$-quadruple and $a<b<c<d$. If one of the following conditions hold
\begin{enumerate}[i)]
\item if $b<2a$ and $c\geq 890b^4$ or
\item if $2a\leq b\leq 12 a$ and $c\geq 1613b^4$ or
\item if $b>12a$ and $c\geq 39247 b^4$ 
\end{enumerate}
then we must have $d=d_+$.
\end{theorem}

\begin{theorem}\label{teorem1.6}
Let $\{a,b,c,d\}$ be a $D(4)$-quadruple and $a<b<c<d_+<d$. Then any $D(4)$-quadruple $\{e,a,b,c\}$ with $e<c$ must be regular.
\end{theorem}

For a fixed $D(4)$-triple $\{a,b,c\}$, denote by $N$ the number of positive integers $d>d_+$ such that $\{a,b,c,d\}$ is a $D(4)$-quadruple. The next theorem is proven as in \cite{cfm}, and similar methods yielded analogous results.

\begin{theorem}\label{teorem_prebrojavanje}
Let $\{a,b,c\}$ be a $D(4)$-triple with $a<b<c$.
\begin{enumerate}[i)]
\item If $c=a+b+2r$, then $N\leq 3$.
\item If $a+b+2r\neq c<b^2$, then $N\leq 7$.
\item If $b^2<c<39247b^4$, then $N\leq 6$.
\item If $c\geq 39247b^4$, then $N=0$.
\end{enumerate}
\end{theorem}
This theorem implies the next corollary.
\begin{kor}
Any $D(4)$-triple can be extended to a $D(4)$-quadruple with $d>\max\{a,b,c\}$  in at most $8$ ways. A regular $D(4)$-triple $\{a,b,c\}$ can be extended to a $D(4)$-quadruple with $d>\max\{a,b,c\}$ in at most $4$ ways. 
\end{kor}
 
 We remark that we can apply the previous results on a family of triples $c$ which arises from the Pellian equation (\ref{par_trojka}) for the fundamental solutions $(t_0,s_0)=(\pm 2,2)$  which gives us a slightly better estimate on a number of extensions when $b<6.85a$. 
 
\begin{proposition}\label{prop_covi}
  Let $\{a,b\}$ be a $D(4)$-pair with $a<b$. Let $c=c^{\pm}_{\nu}$ be given by
  $$
\resizebox{\textwidth}{!}{$c=c_{\nu}^{\pm}=\frac{4}{ab}\left\{\left(\frac{\sqrt{b}\pm\sqrt{a}}{2}\right)^2\left(\frac{r+\sqrt{ab}}{2}\right)^{2\nu}+\left(\frac{\sqrt{b}\mp\sqrt{a}}{2}\right)^2\left(\frac{r-\sqrt{ab}}{2}\right)^{2\nu}-\frac{a+b}{2}\right\}$}
$$
where $\nu\in\mathbb{N}$. 
\begin{enumerate}[i)]
\item If $c=c_1^{+}$ or $c=c_1^-$, then $N\leq 3$.
\item If $c_2^+\leq c\leq c_4^+$ then $N\leq 6$.
\item If $c=c_2^-$ and $a\geq 2$ then $N\leq 6$ and if $a=1$ then $N\leq 7$.
\item If $c\geq c_5^-$ or $c\geq c_4^-$ and $a\geq 35$ then $N=0$. 
\end{enumerate}
  \end{proposition}

 \begin{kor}\label{kor2}
  Let $\{a,b,c\}$ be a $D(4)$-triple. If $a<b\leq6.85a$ then $N\leq 6$.
  \end{kor}

\section{Preliminary results about elements of a $D(4)$-$m$-tuple}

We begin this section by listing some known results.

\begin{lemma}\label{c_granice}
Let $\{a,b,c\}$ be a $D(4)$-triple and $a<b<c$. Then $c=a+b+2r$ or $c>\max\{ab+a+b,4b\}$.
\end{lemma}
\begin{proof}
 This follows from \cite[Lemma 3]{fil_xy4} and \cite[Lemma 1]{duj}.
\end{proof}

The next lemma can be deduced similarly as \cite[Lemma 2]{petorke}.

\begin{lemma}\label{b10na5}
Let $\{a,b,c,d\}$ be a $D(4)$-quadruple such that $a<b<c<d_+<d$. Then $b>10^5$.
\end{lemma}
\begin{proof}
This result extends the result from \cite[Lemma 2.2]{nas} and \cite[Lemma 3]{bf_parovi}, and is established similarly by using the Baker-Davenport reduction method as described in \cite{dujpet}. For the computation we have used the Mathematica 11.1 software package on a computer with the following specifications; Intel(R) Core(TM) i7-4510U CPU @2.00-3.10 GHz processor. The computation took approximately 170 hours in order to check all of the possibilities.
\end{proof}

From \cite[Theorem 1.1]{fil_par} we have a lower bound on $b$ in the terms of the element $a$.
\begin{lemma}\label{b_a_57sqrt}
If  $\{a,b,c,d\}$ is a $D(4)$-quadruple such that $a<b<c<d_+<d$ then $b\geq a+57\sqrt{a}$.
\end{lemma}

The next lemma gives us possibilities for the initial terms of the sequences $v_m$ and $w_n$, and will be improved by Theorem \ref{teorem_izbacen_slucaj}.

\begin{lemma}\cite[Lemma 9]{fil_xy4}\label{lemma_fil_pocetni}
Suppose that $\{a,b,c,d\}$ is a $D(4)$-quadruple with $a<b<c<d$ and that  $w_m$ and  $v_n$  are defined as before. 
\begin{enumerate}
\item[i)] If  equation $v_{2m}=w_{2n}$ has a solution, then $z_0=z_1$ and $|z_0|=2$ or $|z_0|=\frac{1}{2}(cr-st)$ or $z_0<1.608a^{-5/14}c^{9/14}$.
\item[ii)] If equation $v_{2m+1}=w_{2n}$ has a solution, then $|z_0|=t$, $|z_1|=\frac{1}{2}(cr-st)$ and $z_0z_1<0$.
\item[iii)] If equation $v_{2m}=w_{2n+1}$ has a solution, then $|z_1|=s$, $|z_0|=\frac{1}{2}(cr-st)$ and $z_0z_1<0$.
\item[iv)] If equation $v_{2m+1}=w_{2n+1}$ has a solution, then $|z_0|=t$, $|z_1|=s$ and $z_0z_1>0$.
\end{enumerate}
\end{lemma}

\begin{rem}
From the proof of \cite[Lemma 9]{fil_xy4} we see that the case where $v_{2m}=w_{2n}$ and $z_0<1.608a^{-5/14}c^{9/14}$ holds only when $d_0=\frac{z_0^2-4}{c}$, $0<d_0<c$, is such that $\{a,b,d_0,c\}$ is an irregular $D(4)$-quadruple. As we can see from the statement of Theorem \ref{teorem_izbacen_slucaj}, this case will be proven impossible and only cases where $z_0$ is given in the terms of elements of a triple $\{a,b,c\}$ will remain.   
\end{rem}

Using the lower bound on $b$ in an irregular quadruple from Lemma \ref{b10na5} we can slightly improve \cite[Lemma 1]{fil_xy4}.

\begin{lemma}\label{granice_fundamentalnih}
Let $(z,x)$ and $(z,y)$ be positive solutions of (\ref{prva_pelova_s_a}) and (\ref{druga_pelova_s_b}). Then there exist solutions $(z_0,x_0)$ of (\ref{prva_pelova_s_a}) and $(z_1,y_1)$ of (\ref{druga_pelova_s_b}) in the ranges
\begin{align*}
1&\leq x_0<\sqrt{s+2}<1.00317\sqrt[4]{ac},\\
1&\leq |z_0|<\sqrt{\frac{c\sqrt{c}}{\sqrt{a}}}<0.05624c,\\
1&\leq y_1<\sqrt{t+2}<1.000011\sqrt[4]{bc},\\
1&\leq |z_1|<\sqrt{\frac{c\sqrt{c}}{\sqrt{b}}}<0.003163c,
\end{align*}
such that 
\begin{align*}
z\sqrt{a}+x\sqrt{c}&=(z_0\sqrt{a}+x_0\sqrt{c})\left(\frac{s+\sqrt{ac}}{2}\right)^m,\\
z\sqrt{b}+y\sqrt{c}&=(z_1\sqrt{b}+y_1\sqrt{c})\left(\frac{t+\sqrt{bc}}{2}\right)^n.
\end{align*}
\end{lemma}

The statement of the next lemma follows the notations from Theorem \ref{teorem_izbacen_slucaj}.
\begin{lemma}\label{regularna_rjesenja}
Let $\{a,b,c\}$ be a $D(4)$-triple such that $z=v_m=w_n$ has a solution $(m,n)$ for which  $d=d_+=\frac{z^2-4}{c}$. Then only one of the following cases can occur:
\begin{enumerate}[i)]
\item $(m,n)=(2,2)$ and $z_0=z_1=\frac{1}{2}(st-cr)$,
\item $(m,n)=(1,2)$ and $z_0=t$, $z_1=\frac{1}{2}(st-cr)$,
\item $(m,n)=(2,1)$ and $z_0=\frac{1}{2}(st-cr)$, $z_1=s$,
\item $(m,n)=(1,1)$ and $z_0=t$, $z_1=s$.
\end{enumerate}
\end{lemma} 
\begin{proof}
If $n\geq 3$ then 
\begin{align*}
z&\geq w_3> \frac{c}{2y_1}(t-1)^2>\frac{c}{2.000022\sqrt[4]{bc}}\cdot 0.999bc\\
&>0.499b^{3/4}c^{7/4}>0.499bc^{6/4}>157bc
\end{align*}
where we have used Lemma \ref{granice_fundamentalnih} and $bc>10^{10}$ from Lemma \ref{b10na5}.
On the other hand, when $d=d_+$ we have $z=\frac{1}{2}(cr+st)<cr<cb$, which is an obvious contradiction. 
So, when $d=d_+$, we must have $n\leq 2$. Also, since $a<b<c<d$, from the proof of \cite[Lemma 6]{fil_xy4} we see that when $n\leq 2$, the only possibility is $d=d_+$ when $(m,n;z_0,z_1)\in\{ (1,1;t,s),(1,2;t,\frac{1}{2}(st-cr)),(2,1;\frac{1}{2}(st-cr),s),(2,2;\frac{1}{2}(st-cr),\frac{1}{2}(st-cr))\}$.
\end{proof}

This lemma can now be used to get a lower bound on $d$ in terms of the elements of a triple $\{a,b,c\}$.
\begin{lemma}\label{lemaw4w5w6}
Suppose that $\{a,b,c,d\}$ is a $D(4)$-quadruple with $a<b<c<d_+<d$ and that $v_m$ and  $w_n$ are defined as before.
If $z\geq w_n$, $n=4,5,6,7$, then 
$$d>k\cdot b^{n-1.5}c^{n-0.5}$$
where $k=0.249979,0.249974,0.249969,0.249965$ respectively.

If $z\geq v_m$, $n=4,5,6$, then 
$$d>l\cdot a^{m-1.5}c^{n-0.5}$$
where $l=0.243775,0.242245,0.240725$ respectively.

\end{lemma}
\begin{proof}
In the proof of \cite[Lemma 5]{fil_xy4} it has been shown that
\begin{eqnarray*}
\frac{c}{2x_0}(s-1)^{m-1}&<v_m&<cx_0s^{m-1},\\
\frac{c}{2y_1}(t-1)^{n-1}&<w_n&<cy_1t^{n-1}.
\end{eqnarray*}
We use $bc>10^{10}$ for $d>d_+$ and $d= \frac{z^2-4}{c}$ to obtain the desired inequalities.
\end{proof}

We know a relation between $m$ and $n$ if $v_m=w_n$.
\begin{lemma}\cite[Lemma 5]{fil_xy4}\label{mnfilipin}
Let $\{a,b,c,d\}$ be $D(4)$-quadruple. If $v_m=w_n$ then $n-1\leq m\leq 2n+1$.
\end{lemma}

We remark that we can also prove that a better upper bound holds using a more precise argument and the fact that $c<7b^{11}$ (which is proven in \cite[Lemma 8]{sestorka}).
\begin{lemma}\label{lema_epsilon_m_n}
If $c>b^{\varepsilon}$, $1\leq \varepsilon< 12$, then
$$m<\frac{\varepsilon+1}{0.999\varepsilon}n+1.5-0.4\frac{\varepsilon+1}{0.999\varepsilon}.$$
\end{lemma}
\begin{proof}
If $v_m=w_n$ then
$$2.00634^{-1}(ac)^{-1/4}(s-1)^{m-1}<1.000011(bc)^{1/4}t^{n-1}.$$
Since $c>b>10^5$ we can easily check that $s-1>0.9968a^{1/2}c^{1/2}$ and $b^{1/2}c^{1/2}<t<1.0001b^{1/2}c^{1/2}$, so the previous inequality implies
$$(s-1)^{m-3/2}<2.00956t^{n-0.5}<t^{n-0.4}.$$
On the other hand, the assumption $b<c^{1/\varepsilon}$ yields $t<1.0001c^{\frac{\varepsilon +1}{2\varepsilon}}$ and
$$s-1>0.9968c^{1/2}>\left(0.99t \right)^{\frac{\varepsilon}{\varepsilon+1}}>t^{0.999\frac{\varepsilon}{\varepsilon+1}}.$$
So we observe that an inequality 
$$0.999(m-1.5)\frac{\varepsilon}{\varepsilon +1}<n-0.4$$
must hold, which proves our statement. 
\end{proof}

The next lemma follows from the results of \cite[Lemma 5]{ireg_pro} and \cite[Lemma 5]{sestorka}.
\begin{lemma}\label{donje na m i n}
If $v_m$=$w_n$ has a solution for $m>2$, then $6\leq m\leq 2n+1$ and $n\geq 7$ or the case $|z_0|<1.608a^{-5/14}c^{9/14}$ from Lemma \ref{lemma_fil_pocetni} holds and either $6\leq m\leq 2n+1$ or $m=n=4$.
\end{lemma}

\begin{lemma}\label{lema_d_vece}
Assume that $c\leq 0.243775a^{2.5}b^{3.5}$.  If $z=v_m=w_n$ for some even $m$ and $n$ then $d>0.240725a^{4.5}c^{5.5}$.
\end{lemma}
\begin{proof}
Let us assume that $z_0\notin\left\{2,\frac{1}{2}(cr-st)\right\}$, i.e.\@ there exists an irregular $D(4)$-quadruple $\{a,b,d_0,c\}$, $c>d_0$ but then Lemma \ref{lemaw4w5w6} implies $c>0.243775a^{2.5}b^{3.5}$, a contradiction. So, we must have $z_0\in\left\{2,\frac{1}{2}(cr-st)\right\}$ and by Lemma \ref{donje na m i n} we know that $\max\{m,n\}\geq 6$. The statement now follows from Lemma \ref{lemaw4w5w6}. 
\end{proof}

By using the improved lower bound on $d$ in an irregular quadruple from Lemma \ref{lemaw4w5w6} we can prove the next result in the same way as \cite[Lemma 9]{fil_xy4} is proved.
\begin{lemma}
If $v_{m}=w_{n}$ for some even $m$ and $n$ and $|z_0|\notin \left\{2,\frac{1}{2}(cr-st)\right\}$ then $|z_0|<1.2197b^{-5/14}c^{9/14}$.
\end{lemma}
We can prove some upper and lower bounds on $c$ in the terms of smaller elements, depending on the value of $z_0$, similarly as in \cite{fm}, by using Lemma \ref{granice_fundamentalnih}
\begin{lemma}\label{lema_tau}
Set
$$\tau=\frac{\sqrt{ab}}{r}\left(1-\frac{a+b+4/c}{c} \right), \quad (<1).$$
We have that
\begin{enumerate}[i)]
\item $|z_0|=\frac{1}{2}(cr-st)$ implies $c<ab^2\tau^{-4}$,
\item $|z_1|=\frac{1}{2}(cr-st)$ implies $c<a^2b\tau^{-4}$,\label{2.8drugi}
\item $|z_0|=t$ implies $c>ab^2$, \label{2.8treci}
\item $|z_1|=s$ implies $c>a^2b$,
\end{enumerate}
and \ref{2.8drugi}) and \ref{2.8treci}) cannot occur simultaneously when $d>d_+$.
\end{lemma}

The next lemma follows easily  by induction.
\begin{lemma}\label{lema_nizovi}
Let $\{v_{z_0,m}\}$ denote a  sequence $\{v_m\}$ with an initial value $z_0$ and $\{w_{z_1,n}\}$ denote a sequence $\{w_n\}$ with an initial value $z_1$. It holds that $v_{\frac{1}{2}(cr-st),m}=v_{-t,m+1}$, $v_{-\frac{1}{2}(cr-st),m+1}=v_{t,m}$ for each $m\geq 0$ and $w_{\frac{1}{2}(cr-st),n}=w_{-s,n+1}$, $w_{-\frac{1}{2}(cr-st),n+1}=w_{s,n}$ for each $n \geq 0$.
\end{lemma}

For the proof of Theorem \ref{teorem_prebrojavanje} we will use the previous lemma. It is obvious that if we "shift" a sequence as in Lemma \ref{lema_nizovi} the initial term of the new sequence would not necessarily satisfy the bounds from Lemma \ref{granice_fundamentalnih}. In the next lemma we will prove some new lower bounds on $m$ and $n$ when $|z_0|=t$ and $|z_1|=s$, but without assuming the bounds on $z_0$ and $z_1$ from Lemma \ref{granice_fundamentalnih}. Since Filipin has deduced in \cite{ireg_pro} that $n\geq 7$ in any case $(z_0,z_1)$ which appears when "shifting" sequences as in Lemma \ref{lema_nizovi}, we can consider that bound already established. Even though the following proof is analogous to the proof of \cite[Lemma 2.6]{cfm}, there is more to consider in the $D(4)$ case which is why we present some details of the proof.

\begin{proposition}\label{prop_m_n_9}
Let $\{a,b,c\}$ be a $D(4)$-triple, $a<b<c$, $b>10^5$ and $c>ab+a+b$. Let us assume that the equation $v_m=w_n$ has a solution for $m>2$ such that $|z_0|=t$, $|z_1|=s$, $z_0z_1>0$ and $m\equiv n\equiv 1 \ (\bmod \ 2)$. Then $\min\{m,n\}\geq 9$.
\end{proposition}
\begin{proof}
It is easy to see that $v_1=w_1<v_3<w_3$. If we show that
\begin{enumerate}[i)]
\item $w_3<v_5<w_5<v_9$ and $v_7\neq w_5$,
\item $v_7<w_7<v_{13}$ and $v_9\neq w_7\neq v_{11}$,
\end{enumerate}
from Lemma \ref{mnfilipin} we see that it leads to a conclusion that $\min\{m,n\}\geq 9$. 

Let $(z_0,z_1)=(\pm t,\pm s)$. We derive that
\begin{eqnarray*}
w_3&=&\frac{1}{2}(cr\pm st)(bc+1)+cr\\
w_5&=&\frac{1}{2}(cr\pm st)(b^2c^2+3bc+1)+cr(bc+2)\\
w_7&=&\frac{1}{2}(cr\pm st)(b^3c^3+5b^2c^2+6bc+1)+cr(b^2c^2+4bc+3)
\end{eqnarray*}
and
\begin{eqnarray*}
v_5&=&\frac{1}{2}(cr\pm st)(a^2c^2+3ac+1)+cr(ac+2)\\
v_7&=&\frac{1}{2}(cr\pm st)(a^3c^3+5a^2c^2+6ac+1)+cr(a^2c^2+4ac+3)\\
v_9&=&\frac{1}{2}(cr\pm st)(a^4c^4+7a^3c^3+15a^2c^2+10ac+1)+\\&&cr(a^3c^3+6a^2c^2+10ac+4)\end{eqnarray*}
\begin{eqnarray*}
v_{11}&=&\frac{1}{2}(cr\pm st)(a^5c^5+9a^4c^4+28a^3c^3+35a^2c^2+15ac+1)+\\&&cr(a^4c^4+8a^3c^3+21a^2c^2+20ac+5)\\
v_{13}&=&\frac{1}{2}(cr\pm st)(a^6c^6+11a^5c^5+45a^4c^4+84a^3c^3+70a^2c^2+21ac+1)+\\&&cr(a^5c^5+10a^4c^4+36a^3c^3+56a^2c^2+35ac+6)
\end{eqnarray*}
Since $a<b$, $c>ab$, and the sequences $v_m$ and $w_n$ are increasing, it is easy to see that $w_3<v_5<v_7$, $v_5<w_5<v_7<v_9<v_{11}$, $v_7<w_7$, $w_5<v_9$ and $w_7<v_{13}$. It remains to prove that $v_7\neq w_5$, $v_9\neq w_7$ and $ v_{11}\neq w_7$.

From the Lemma \ref{donje na m i n} and the explanation before this lemma, we consider that $v_7\neq w_5$ is already proven but it is not hard to follow the next proof in order to deduce this case also. We will show only that $v_9\neq w_7$, and the case  $v_{11}\neq w_7$ can be shown similarly with only some technical details changed.

\medskip
 Let us assume to the contrary, that $v_9= w_7$. We have for $(z_0,z_1)=(\pm t,\pm s)$ that
\begin{align*}
&cr(a^4c^4+9a^3c^3+27a^2c^2+30ac-b^3c^3-7b^2c^2-14bc+2)=\\
&\mp st (a^4c^4+7a^3c^3+15a^2c^2+10ac-b^3c^3-5b^2c^2-6bc).
\end{align*}

\begin{enumerate}[a)]
\item \textbf{Case $z_0=-t$}\\
Since $cr>st$ we have
\begin{align*}
&a^4c^4+9a^3c^3+27a^2c^2+30ac-b^3c^3-7b^2c^2-14bc+2<\\
&<a^4c^4+7a^3c^3+15a^2c^2+10ac-b^3c^3-5b^2c^2-6bc
\end{align*}
which leads to 
\begin{equation}\label{v9w7_prva_njdk}a^3c<b^2.
\end{equation}
Since $c>10^5$ we have $b>316a$. On the other hand, we easily see that 
$$cr-st>\frac{4c^2-4ac-4bc-16}{2cr}>\frac{2ab}{r}>1.99r,$$
which can be used to prove
\begin{align*}
v_9&<\frac{1}{2}(cr- st)(a^4c^4\left(2.0051+\frac{7}{ac} \right)\\&+15a^2c^2+10ac+1)+cr(6a^2c^2+10ac+4),\\
w_7&>\frac{1}{2}(cr- st)(b^3c^3+15a^2c^2+10ac+1)+cr(6a^2c^2+10ac+4).
\end{align*}
So 
\begin{equation}\label{v9w7_druga} b^3<2.006a^4c.
 \end{equation}
 By combining equations (\ref{v9w7_prva_njdk}) and (\ref{v9w7_druga}) we see that 
 $b<2.006a$, which is in a contradiction to $b>316a$. 

\item \textbf{Case $z_0=t$}\\
Since $cr>st$ we have
\begin{align*}
&a^4c^4+9a^3c^3+27a^2c^2+30ac-b^3c^3-7b^2c^2-14bc+2<\\
&<-a^4c^4-7a^3c^3-15a^2c^2-10ac+b^3c^3+5b^2c^2+6bc
\end{align*}
which leads to 
\begin{equation}
2a^4c^4+16a^3c^3+42a^2c^2+40ac+2<2b^3c^3+12b^2c^2+20bc.
\end{equation}
If $16a^3c^3<12b^2c^2$ then $c<0.75\frac{b^2}{a^3}$. On the other hand, if $16a^3c^3\geq 12b^2c^2$ then 
$2a^4c^4<2b^3c^3$. In each case, inequality  
\begin{equation}\label{v9w7_tplus_prva_njdk}a^4c<b^3
\end{equation}
holds. Since $c>10^5$, $b>46a$ and since $c>ab$ we have $b>a^{5/2}$ and $c>a^{7/2}$.
It is easily shown that $cr+st>632r^2$. We can use this to see that
\begin{align*}
\frac{1}{2}(cr+ st)(a^4c^4+7a^3c^3)+cr\cdot a^3c^3&<\frac{1}{2}(cr+ st)a^4c^4\left(1+\frac{2}{632ar}+\frac{7}{ac} \right),
\end{align*}
from which we get an upper bound on $v_9$. Now, from $v_9=w_7$ we have
\begin{equation}\label{v9w7_tplus_druga_njdk}b^3<1.001a^4c.
\end{equation}
Notice that $c>b^3 1.001^{-1}a^{-4}>0.999b^{7/5}>9.99\cdot 10^6$.

If we consider $v_9=w_7$ modulo $c^2$ and use the fact that $st(cr-st)\equiv 16 \ (\bmod \ c)$ it yields a congruence 
$$2r(cr-st)\equiv 16(6b-10a)\ (\bmod \ c).$$
Since $2r(cr-st)<4.01c$ and $16(6b-10a)<96b<96\cdot (1.001a^4c)^{1/3}<97c^{5/7}<c$, we have that one of the equalities
$$kc=2r(cr-st)-16(6b-10a),\quad k\in\{0,1,2,3,4\}$$
must hold. 

If $k=0$, we have $2r(cr-st)=16(6b-10a)$. The inequality $2r(cr-st)>3.8r^2>3.8ab$ implies $96b>16(6b-10a)>3.8ab$ so $a\leq 25 $. Now, we have
\begin{align*}
cr-st&>\frac{2c^2-2ac-2bc-8}{cr}>\frac{2\frac{b^3}{25^4\cdot 1.001}-(2b+50+1)}{\sqrt{25b+4}}\\
&>\frac{b(0.00000255b^2-2)-51}{b^{1/2}\sqrt{26}}>0.0000005b^{5/2},
\end{align*}
so $16(6b-10a)>2\sqrt{ab+4}\cdot 0.0000005b^{5/2}$. For each $a\in[1,25]$, we get from this inequality a numerical upper bound on $b$ which is in a contradiction to $b>10^5$.

If $k\neq 0$, we have a quadratic equation in $c$ 
with possible solutions
$$c_{\pm}=\frac{-B\pm\sqrt{B^2-4AE}}{2A}$$
where $A=r^2(16-4k)+k^2>0$, $B=-(32(2r^2-6)(6b-10a)+16r^2(a+b))<0$ and $E=64(4(6b-10a)^2-r^2)>0$.

If $k\leq 3$, $A> 4r^2$ and $c_{\pm}<\frac{-B}{A}<\frac{32\cdot2r^2+6b+16r\cdot32b}{4r^2}=100b$. Since $c>ab+a+b$, we have $a\leq 98$ and from $b^3<1.001a^4c<100.1a^4b$ we have $b<100.1^{1/2}a^2<96089$ which is in a contradiction to $b>10^5$.

In only remains to check if $k=4$. But in this case we express $b$ in the terms of $a$ and $c$ and get
$$b_{\pm}=\frac{-B\pm\sqrt{B^2+4AE}}{2A}$$
where $A=400ac-9216>0$, $B=c(16a^2-640a+832)+64a+30720$ and $E=16c^2+1216ac+25600a^2-256>0$.

We have $B^2+4AE>4AE>25536ac^3$, so it is not hard to see that $b_{-}<0$. Also, $B<0$ when $a\in[2,38]$, and $B>0$ otherwise. When $B>0$, $b_+<\frac{\sqrt{B^2+4AE}}{2A}<\frac{\sqrt{256a^4c^2+400ac^2(16c+2000a)}}{798ac}$ and since $a<c^{2/7}$, we get $b_+<\frac{84a^{1/2}c^{3/2}}{798ac}<0.106c^{1/2}a^{-1/2}$. On the other hand, $b_+>\frac{\sqrt{25536ac^3}-16a^2c}{800ac}>0.18c^{1/2}a^{-1/2}$, which is a contradiction to the previous inequality. 
In the last case, when $B<0$, we have $|B|<7232c$ and get a similar contradiction.\qedhere
\end{enumerate}
\end{proof}

\section{Proof of Theorem \ref{teorem_izbacen_slucaj}}
 
\begin{lemma}\label{lema_uvodna_teorem_izbacen}
Let us assume that $c>0.243775a^2b^{3.5}$ and that $z=v_m=w_n$ has a solution for $m$ and $n$, $n\geq 4$. Then $m\equiv n \ (\bmod \ 2)$ and
\begin{enumerate}[i)]
\item if $m$ and $n$ are even and $b\geq 2.21a$, then $$n>0.45273b^{-9/28}c^{5/28},$$
and if $b<2.21a$ then 
$$n>\min \{0.35496a^{-1/2}b^{-1/8}c^{1/8},0.177063b^{-11/28}c^{3/28}\};$$
\item if $m$ and $n$ are odd, then
$$n>0.30921b^{-3/4}c^{1/4}.$$
\end{enumerate}
\end{lemma}
\begin{proof}
Since from $b>10^5$ and $c>0.243775a^2b^{3.5}$ we have $c>ab^2$ and $\tau^{-4}<b/a$, and cases $i)$ and $ii)$ from Lemma \ref{lema_tau} cannot hold. So, we see that the only options from Lemma \ref{lemma_fil_pocetni} are $i)$, when $|z_0|=2$ or $|z_0|<1.219b^{-5/14}c^{9/14}$, and $iv)$. In each case we have $m\equiv n \ (\bmod \ 2)$.

First, let us consider the case  $|z_0|<1.219b^{-5/14}c^{9/14}$, $z_0=z_1$ and $m$ and $n$ even. Since $(z_0,x_0)$ satisfies an equality
$$cx_0^2-az_0^2=4(c-a)$$
we have 
$$x_0^2\leq \frac{a}{c}\cdot 1.2197^2b^{-5/7}c^{9/7}+4\left(1-\frac{a}{c} \right)<1.7109ac^{2/7}b^{-5/7},$$
where we have used the estimate $c^2b^{-5}>0.243775^2b^2>5.9426\cdot 10^8$. So,
\begin{equation}\label{gornja_xo}
x_0<1.30802a^{1/2}b^{-5/14}c^{1/7}.
\end{equation}
Similarly, we get 
\begin{equation}\label{gornja_y1}
y_1<1.21972b^{1/7}c^{1/7}.
\end{equation}
From \cite[Lemma 12]{fil_xy4} we have that the next congruence holds
\begin{equation} \label{kongruencija} az_0m^2-bz_1n^2\equiv ty_1n-sx_0m\ (\bmod \ c).
\end{equation}

From $b>10^5$ and $c>0.243775b^{3.5}$ we have $c>b^{3.377}$, so we can use $\varepsilon=3.377$ in the inequality from Lemma \ref{lema_epsilon_m_n} and get 
$$m<1.2975n+0.9811.$$
This implies that the inequality $m<1.34n$ holds for every possibility $m$ and $n$ even, except for $(m,n)=(6,4)$, which we will study separately. 

Now we study the case where $b\geq 2.21a$ and we assume to the contrary, that $n\geq 0.45273b^{-9/28}c^{5/28}$. Then from $c>b>10^5$ we have
\begin{align*}
am^2|z_0|&<a\cdot 1.34^2\cdot 0.45273^2 b^{-9/14}c^{5/14}\cdot 1.2197b^{-5/14}c^{9/14}\\
&<a\cdot 0.4489b^{-1}c<\frac{c}{4},
\end{align*}
\begin{align*}
bn^2|z_0|&<b\cdot 0.45279^2\cdot b^{-9/14}c^{5/14}\cdot 1.2197b^{-5/14}c^{9/14}<\frac{c}{4},
\end{align*}
and from inequalities (\ref{gornja_xo}) and (\ref{gornja_y1}) we also have
\begin{align*}
ty_1n&<(bc+4)^{1/2}1.2197b^{1/7}c^{1/7}\cdot 0.45273b^{-9/28}c^{5/28}\\
&<\frac{2.21b^{1/2}c^{1/2}}{c^{19/28}b^{5/28}}\cdot \frac{c}{4}=\frac{2.21}{c^{5/28}b^{-9/28}}\cdot \frac{c}{4}<\frac{2.85}{a^{5/14}b^{17/56}}\frac{c}{4}<\frac{c}{4}, 
\end{align*}
and similarly
\begin{align*}
sx_0m&<(ac+4)^{1/2}\cdot 1.34n\cdot x_0<0.79361ab^{-19/28}c^{23/28}<\frac{c}{4}.
\end{align*}
In the case $(m,n)=(6,4)$ we can prove the same final inequalities. So, from congruence (\ref{kongruencija}) we see that the equation
\begin{equation}\label{eq1}
az_0m^2+sx_0m=bz_0n^2+ty_1n
\end{equation}
must hold.

On the other hand, from equation $az_0^2-cx_0^2=4(a-c)$, since $|z_0|\neq 2$ in this case, and $c\mid (z_0^2-4)$ we have $z_0^2\geq c+4$. Let us assume that $z_0^2<\frac{5c}{a}$, then we would have $c(x_0^2-9)+4a<0$ and since $c>4a$ we must have $x_0=2$ and $|z_0|=2$, which is not our case. So, here we have $z_0^2\geq \max\{c+4,\frac{5c}{a}\}$.

Now, 
\begin{align*}
0\leq \frac{sx_0}{a|z_0|}-1&=\frac{4x_0^2+4ac-4a^2}{a|z_0|(sx_0+a|z_0|)}\leq \frac{2x_0^2}{a^2z_0^2}+\frac{2ac-2a^2}{a^2z_0^2}\\
&<2\left(\frac{1}{ac}+4\left(1-\frac{a}{c} \right)\frac{1}{a^2z_0^2} \right)+\left( \frac{1}{z_0^2}\left(2\frac{c}{a}-2 \right)\right)\\
&<\frac{18}{5ac}+\left(\frac{2}{5}-\frac{2a}{5c}\right)\leq 0.000036+0.4=0.400036, 
\end{align*}
and similarly
$$0\leq \frac{ty_1}{b|z_0|}-1<0.00004.$$

When $z_0>0$, i.e.\@ $z_0>2$, we have 
$$bz_0n(n+1)<bz_0n(n+\frac{ty_1}{bz_0})=az_0m(m+\frac{sx_0}{az_0})< az_0m(m+1.400036).$$

Since $m\geq 4$ we have from the previous inequality that
$2.21n(n+1)<1.35001m^2<1.35001(1.2975n+0.9811)$ must hold. But then we get $n<1$, an obvious contradiction. 

On the other hand, if $z_0<0$, i.e. $z_0<-2$, we similarly get
\begin{equation}\label{njdk:1}
am(m-1)>bn(n-1.00004).\end{equation}
Since, $m<1.2975n+0.9811$ and $b>2.21a$ we have $n\leq 6$.
If $b\geq 2.67a$, we would have $n<4$. So, it only remains to study the case $2.21a\leq b<2.67a$ and $n\leq6$. In that case we have $c>0.243775a^2c^{3.5}>0.03419b^{5.5}$, so we can put $\varepsilon=5.2$ in Lemma \ref{lema_epsilon_m_n} and get $m<1.1936n+1.0226$. Inserting in the inequality (\ref{njdk:1}) yields that only the case $(m,n)=(4,4)$ remains, but it doesn't satisfy equation (\ref{njdk:1}).

Now, let us study the case where $b<2.21a$. 
We again consider the congruence (\ref{kongruencija}), which after squaring and using $z_0^2\equiv t^2 \equiv s^2 \equiv 4 \ (\bmod \ c)$ yields a congruence
\begin{equation}\label{kongr2}
4((am^2-bn^2)^2-y_1^2n^2-x_0^2m^2)\equiv -2tsx_0y_1mn\ (\bmod \ c).
\end{equation} 
Let us denote $C=4((am^2-bn^2)^2-y_1^2n^2-x_0^2m^2)$, and (\ref{kongr2}) multiplied by $s$ and by $t$ respectively shows that
\begin{eqnarray}
Cs&\equiv& -8tx_0y_1mn \ (\bmod \ c)\label{cs}\\
Ct&\equiv& -8sx_0y_1mn \ (\bmod \ c).\label{ct}
\end{eqnarray}

Now, assume that $n\leq \min \{0.35496a^{-1/2}b^{-1/8}c^{1/8}, 0.177063b^{-11/28}c^{3/28}\}.$ Then also $n\leq 0.45273b^{-9/28}c^{5/28}$ holds, so we again have an equality in the congruence (\ref{kongruencija}), i.e.\@
$$az_0m^2+sx_0m=bz_0n^2+ty_1n.$$
We have that $x_0^2<y_1^2<1.4877b^{2/7}c^{2/7}.$ Since $b<2.21a$ we have $c>0.04991b^{5.5}$so we can take $\varepsilon=5.239$ in Lemma \ref{lema_epsilon_m_n} and we get $m<1.1921n+1.0232$, and since we know $m,n\geq 4$ and $m$ and $n$ even, we also have $m<1.34n$ from this inequality. This yields
\begin{align*}
|Cs|<|Ct|&=|4t((am^2-bn^2)^2-(y_1^2n^2+x_0^2m^2))|<\max\{4tb^2m^4,8ty_1^2m^2\}\\
&<\max\{12.896718b^2n^2\sqrt{bc+4},21.370513b^{2/7}c^{2/7}n^2\sqrt{bc+4}\}.
\end{align*}
On the other hand, we have from our assumption on $n$ that 
$$12.896718b^2n^2\sqrt{bc+4}<12.896718\left(\frac{b}{a}\right)^2 0.35496^4\cdot 1.000001c^{1/2}<c,$$
and
$$21.370513b^{2/7}c^{2/7}n^2\sqrt{bc+4}<21.370513\cdot0.177063^2\cdot 1.000001c<c,$$
so, $|Cs|<|Ct|<c$. On the other hand, $8sx_0y_1mn<8tx_0y_1mn$ and
\begin{align*}
8tx_0y_1mn&<8ty_1^2\cdot 1.34n^2\\
&<8\cdot 1.000001b^{1/2}c^{1/2} 1.4877b^{2/7}c^{2/7}\cdot 1.34\cdot 0.177063^2b^{-11/14}c^{3/14}\\
&<0.5c<c,
\end{align*}
So, in congruences (\ref{cs}) and (\ref{ct}) we can only have $$Cs=kc-8tx_0y_1mn,\quad Ct=lc-8sx_0y_1mn,\quad k,l\in\{0,1\}.$$
If $k=l=0$, we would have $s=t$, which is not possible. When $k=l=1$, we get $c=8(t+s)x_0y_1mn<0.5c+0.5c<c$, a contradiction. In the case $k=0$ and $l=1$ we get $cs=8(s^2-t^2)x_0y_1mn<0$, and in the case $k=1$ and $l=0$, we have $ct=8(t^2-s^2)x_0y_1mn$, so $c(t-s)<8(t^2-s^2)x_0y_1mn$, which leads to a contradiction as in the case $k=l=1$.

So, $n> \min \{0.35496a^{-1/2}b^{-1/8}c^{1/8},0.177063b^{-11/28}c^{3/28}\}$ must hold. 

\medskip
It remains to consider the case when $m$ and $n$ are odd. In this case, a congruence from \cite[Lemma 3]{fil_xy4} holds,
\begin{eqnarray}
\pm 2t(am(m+1)-bn(n+1))&\equiv& 2rs(n-m)\ (\bmod \ c),\label{kong3}\\
\pm 2s(am(m+1)-bn(n+1))&\equiv& 2rt(n-m)\ (\bmod \ c).\label{kong4}
\end{eqnarray}
Let us assume that $n\leq 0.30921b^{-3/4}c^{1/4}$. In this case we have $m<1.2975n+0.9811$ and  since $m\geq n\geq 5$ are both odd we also have $m<1.47n$. Notice that
$2t(am(m+1)-bn(n+1))<2tbm(m+1)$ holds. Also, from $c>0.243775b^{3.5}>7.7\cdot 10^{16}$ we have $b<0.243775^{-2/7}c^{2/7}$. So it suffice to observe that  
\begin{align*}
2tbm(m+1)&<2\sqrt{bc}\cdot 1.000001b\cdot 1.21m^2\\
&<5.22944b^{3/2}c^{1/2}<5.22944(0.243775^{-2/7}c^{2/7})^{3/2}c^{1/2}\\
&<7.82711c^{13/14}<0.5c,\\
2rt(m-n)&<2\sqrt{ab}\left(1+\frac{1}{\sqrt{ab}}\right)\sqrt{bc}\left(1+\frac{1}{\sqrt{bc}}\right)(0.2962n+0.85194)^2\\
&<2a^{1/2}bc^{1/2}\cdot 1.00318\cdot 0.51^2 \cdot 0.30921^2 b^{-3/2}c^{1/2}\\
&<0.0499(a/b)^{1/2}c
\end{align*}
which means that we have equalities in congruences (\ref{kong3}) and (\ref{kong4}) and implies 
$$\pm \frac{2rs}{t}(n-m))=\pm \frac{2rt}{s}(n-m). $$
Since $s\neq t$, the only possibility is $n=m$, but then $2t(am(m+1)-bn(n+1))=0$ implies $a=b$, a contradiction. So, our assumption for $n$ was wrong and $n> 0.30921b^{-3/4}c^{1/4}$ when $n$ and $m$ are odd.
\end{proof}

Various versions of Rickert's theorem from \cite{rickert} and results derived from them are often used when considering problems of $D(1)$-$m$-tuples and $D(4)$-$m$-tuples. In this paper we will use a lemma from \cite{bf_parovi} and give a new version of that result, which improves it in some special cases.
 
\begin{lemma}\cite[Lemmas 6 and 7]{bf_parovi}\label{lema_rickert}
Let $\{a,b,c,d\}$ be a $D(4)$-quadruple, $a<b<c<d$ and $c>308.07a'b(b-a)^2a^{-1}$, where $a'=\max\{4a,4(b-a)\}$. Then
$$n<\frac{2\log(32.02aa'b^4c^2)\log (0.026ab(b-a)^{-2}c^2)}{\log(0.00325a(a')^{-1}b^{-1}(b-a)^{-2}c)\log(bc)}.$$
\end{lemma}
By combining results and proofs of \cite[Lemma 7]{bf_parovi} and \cite[Lemma 3.3]{cf}, the next generalization of \cite[Lemma 7]{nas2} can be proved.

\begin{lemma}\label{lema_rickert2}
Let $\{a,b,c,d\}$ be a $D(4)$-quadruple, $a<b<c<d$, $b>10^5$ and $c>59.488a'b(b-a)^2a^{-1}$, where $a'=\max\{4a,4(b-a)\}$. If $z=v_m=w_n$ for some $m$ and $n$ then
$$n<\frac{8\log(8.40335\cdot 10^{13} a^{1/2}(a')^{1/2}b^2c)\log (0.20533a^{1/2}b^{1/2}(b-a)^{-1}c)}{\log(bc)\log(0.016858a(a')^{-1}b^{-1}(b-a)^{-2}c)}.$$

\end{lemma}

\begin{proposition}\label{propozicija_d+}
Let $\{a,b,c,d\}$ be a $D(4)$-quadruple such that $a<b<c<d$ and that equation $z=v_m=w_n$ has a solution for some $m$ and $n$. If $b\geq 2.21a$ and $c>\max\{0.24377a^2b^{3.5},2.3b^5\}$ or $b<2.21a$ and $c>1.1b^{7.5}$ then $d=d_+$.
\end{proposition}
\begin{proof}
Let us assume that $d>d_+$. Since $c>0.24377a^2b^{3.5}$ in both cases, we can use Lemma \ref{lema_uvodna_teorem_izbacen}, and as in the proof of that lemma, $m$ and $n$ have the same parity.  Also, since in any case $c>2.3b^5>308.07a'b(b-a)^2a^{-1}$, we can use Lemmas \ref{lema_rickert} and \ref{lema_rickert2} (note Lemma \ref{lema_rickert} will give better results in these cases).

Let us assume that $m$ and $n$ are even and $b\geq 2.21a$. Then we have $n>0.45273b^{-9/28}c^{5/28}$, $a'=\max\{4a,4(b-a)\}=4(b-a)$, $a\leq \frac{b}{2.21}$ and $b>b-a>\frac{1.21}{2.21}b$. We estimate
\begin{align*}
32.02aa'b^4c^2&=32.02a\cdot4(b-a)b^4c^2<57.955b^6c^2,\\
0.026ab(b-a)^{-2}c^2&<\frac{0.026}{2.21}b^2\frac{2.21^2}{1.21^2}b^{-2}c^2<0.0393c^2,\\
0.00325a(a')^{-1}b^{-1}(b-a)^{-2}c&>0.00325\cdot4^{-1}(b-a)^{-3}b^{-1}c>0.0008125b^{-4}c.
\end{align*} 
Now, from Lemma \ref{lema_rickert} we have an inequality
$$0.45273b^{-9/28}c^{5/28}<\frac{2\log(57.955b^6c^2)\log(0.0393c^2)}{\log(bc)\log(0.0008125b^{-4}c)},$$
where the right-hand side is decreasing in $c$ for $b>0$, $c>2.3b^5$ so we can observe the inequality in which we have replaced  $c$ with $2.3b^5$ which gives us $b<19289$, a contradiction to $b>10^5.$

Let us consider the case when $m$ and $n$ are even and $2a<b<2.21a$. Then $a'=4(b-a)$, $\frac{b}{2}<b-a<\frac{1.21}{2.21}b$. Denote 
\begin{align*}F\in&\{0.35496a^{-1/2}b^{-1/8}c^{1/8},0.177063b^{-11/28}c^{3/28}\}\\&>\{0.50799b^{9/16},0.17888b^{23/56}\}.\end{align*} Then by Lemma \ref{lema_rickert}
$$F<\frac{2\log(35.0627b^6c^2)\log(0.052c^2)}{\log(bc)\log(0.0022397b^{-3}c)}.$$
The right-hand side of this inequality is decreasing in $c$ for $b>0$, $c>1.1b^{7.5}$, and for each possibility for $F$ we get $b<722$ and $b<81874$ respectively, a contradiction in either case.
 
In the case when $m$ and $n$ are even and $b<2a$, we have $a'=4a$ and $57<b-a<\frac{b}{2}$. With $F$ defined as before, we have $F>\{0.35921b^{9/16},0.17888b^{23/56}\}$ and 
$$F<\frac{2\log(128.08b^6c^2)\log(0.0000081b^2c^2)}{\log(bc)\log(0.00325b^{-3}c)}.$$
Again, the right-hand side is decreasing in $c$ for $c>1.1b^{7.5}$. We get $b<1396$ for the first choice for $F$, and $b<98413$ for the second, again, a contradiction. 
\medskip 

It remains to consider the case when $m$ and $n$ are odd. If $b\geq 2a$ we have $a'=4(b-a)$ and $c>2.3b^5$, so similarly as before 
\begin{align*}
32.02aa'b^4c^2&<64.04b^6c^2,\\
0.026ab(b-a)^{-2}c^2&<0.052c^2,\\
0.00325a(a')^{-1}b^{-1}(b-a)^{-2}c&>0.0008125b^{-4}c.
\end{align*}
In this case we have $n>0.30921b^{-3/4}c^{1/4}$ so by Lemma \ref{lema_rickert} we observe an inequality
$$0.30921b^{-3/4}c^{1/4}<\frac{2\log(64.04b^6c^2)\log(0.052c^2)}{\log(bc)\log(0.0008125b^{-4}c)}.$$
Since the right-hand side is decreasing in $c$ for $c>2.3b^5$ we get $b<97144$, a contradiction. 

If $b< 2a$ and $c>1.1b^{7.5}$ we observe
$$0.30921b^{-3/4}c^{1/4}<\frac{2\log(128.08b^6c^2)\log(0.0000081b^2c^2)}{\log(bc)\log(0.00325b^{-3}c)}$$
and since the right-hand side is decreasing in $c$ for $c>1.1b^{7.5}$ we get $b<48$, a contradiction.
\end{proof}

\bigskip
\begin{proof}[Proof of Theorem \ref{teorem_izbacen_slucaj}]
Let us assume that $d=d_+$. Then Lemma \ref{regularna_rjesenja} gives us all of the possible fundamental solutions and indices.

Let us assume that $d>d_+$. Then from Lemma \ref{donje na m i n} we have $n\geq 4$. Here we only consider a possibility when $m$ and $n$ are even and $z_0=z_1\notin \left\{2,\frac{1}{2}(cr-st) \right\}$. Then from \cite{fil_xy4} we know that for $d_0=\frac{z_0^2-4}{c}<c$, $D(4)$-quadruple $\{a,b,d_0,c\}$ is irregular.

Denote $\{a_1,a_2,a_3\}=\{a,b,d_0\}$ such that $a_1<a_2<a_3$. 

If $a_3\leq 0.234775a_1^{2.5}a_2^{3.5}$ holds then by Lemma \ref{lema_d_vece} we also have an inequality $c>0.240725a_1^{4.5}a_3^{5.5}\geq 0.240725a^{4.5}b^{5.5}$. If $b>2.21a$, since $b>10^5$, this inequality implies  $c>76a^{4.5}b^{5.5}$. On the other hand, if $b\leq 2.21a$, we have $a\geq 45249$ and $c>0.240725\cdot 2.21^{-2}a^{2.5}b^{7.5}>10a^{2}b^{7.5}$. We see that in either case we can use Proposition \ref{propozicija_d+} and conclude $d=d_+$, i.e. we have a contradiction to the assumption $d>d_+$.

It remains to consider the case $a_3>0.234775a_1^{2.5}a_2^{3.5}$. If $b=a_2$ then from Lemma \ref{lemaw4w5w6} we get
$$c>0.249979b^{2.5}(0.243775b^{3.5})^{3.5}>0.00178799b^{14.75}.$$
It is easy to see that we again have conditions of Proposition \ref{propozicija_d+} satisfied and can conclude that $d=d_+$ is the only possibility.
On the other hand, if $b=a_3$, then $b>0.243775a_1^{2.5}a_2^{3.5}>0.243775\cdot 10^{17.5}$ since $a_2>10^5$. As such, by Lemma \ref{lema_uvodna_teorem_izbacen} we need to consider two cases. First, when $a_2\geq 2.21a_1 $ then $n'>0.45273a_2^{-9/28}b^{5/28}>0.45273(0.243775^{-2/7}b^{2/7})^{-9/28}b^{5/28}>0.3976b^{17/196}>11$. By Lemma \ref{lemaw4w5w6} it follows that $z'>w_6$ and $c>0.249969a_1^{4.5}b^{5.5}$ which, as before, yields a contradiction when Proposition \ref{propozicija_d+} is applied. The second case is when $a_2<2.21a_1$. If  $n'\geq 6$, we have $c>0.249969a_1^{4.5}b^{5.5}>b^{7.5}$, since $a_1>\frac{a_2}{2.21}>\frac{10^5}{2.21}$ and get the same conclusion as before. If $n<6$, by Lemma \ref{donje na m i n}, we see that we have $m'$ and $n'$ even, so $n'=4$. Since $b>0.0335a_2^6>a_2^{5.7}$ from Lemma \ref{lema_epsilon_m_n} we have that $m'<1.1766n'+1.0294$, i.e. $m'=4$. From the proof of \cite[Lemma 5]{sestorka} we know that $m'=n'=4$ can hold only when $|z_0'|<1.2197(b')^{-5/14}(c')^{9/14}$. As such, we have a $0<d_0'<b$ such that $\{a,d_0',d_0,b\}$ is an irregular $D(4)$-quadruple and we can use the same arguments to prove that such a quadruple cannot exist by Proposition \ref{propozicija_d+}, or we have a new quadruple with $0<d_0''<d_0'<b$. Since this process cannot be repeated infinitely, for some of those quadruples in the finite process we must have $n>6$, a contradiction to Proposition \ref{propozicija_d+}.

The last assertion of Theorem \ref{teorem_izbacen_slucaj} follows from Lemma \ref{lema_tau}.
\end{proof}

\section{Proofs of Theorems \ref{teorem1.5} and \ref{teorem1.6}}
A more general result for the lower bound on $m$ can be established by following an analogous argument as in \cite[Proposition 3.1]{nas} and \cite[Lemma 16]{nas2}.

\begin{lemma}\label{lema_donja_m_staro}
Let $\{a,b,c,d\}$  be a $D(4)$-quadruple with $a<b<c<d$ for which $v_{2m}=w_{2n}$ has a solution such that $z_0=z_1=\pm 2$ and $Ln\geq m \geq n \geq 2$ and $m\geq 3$, for some real number $L>1$. 

Suppose that $a\geq a_0$, $b\geq b_0$, $c\geq c_0$ and $b\geq \rho a$ for some positive integers $a_0$, $b_0$, $c_0$ and a real number $\rho>1$. Then
$$m>\alpha b^{-1/2} c^{1/2},$$
where $\alpha$ is any real number satisfying both inequalities
\begin{equation}\label{njdba_1}
\alpha^2+(1+2b_0^{-1}c_0^{-1})\alpha\leq 1
\end{equation}
\begin{equation}\label{njdba_2}
4\left(1-\frac{1}{L^2} \right)\alpha^2+\alpha(b_0(\lambda + \rho^{-1/2})+2c_0^{-1}(\lambda + \rho^{1/2}))\leq b_0
\end{equation}
with $\lambda=(a_0+4)^{1/2}(\rho a_0+4)^{-1/2}$.
\\Moreover, if $c^\tau \geq \beta  b$ for some positive real numbers $\beta$ and $\tau$ then
\begin{equation}\label{njdba_3}
m>\alpha \beta^{1/2}c^{(1-\tau)/2}.
\end{equation}
\end{lemma}

\begin{lemma}\label{lema_donja_m}
Let us assume that $c>b^4$ and $z=v_m=w_n$  has a solution for some positive integers $m$ and $n$. Then $m\equiv n \ (\bmod \ 2)$ and $n>0.5348b^{-3/4}c^{1/4}$. 
\end{lemma}
\begin{proof}
As in the proof of Lemma \ref{lema_uvodna_teorem_izbacen}, we can see that $m\equiv n \ (\bmod \ 2)$ when $c>b^4$, and when $m$ and $n$ are even, the only possibility is $|z_0|=2$. By  Theorem \ref{teorem_izbacen_slucaj} we need to consider two cases.
\begin{enumerate}[1)]
\item Case $m$ and $n$ are even, $|z_0|=|z_1|=2$. \\
From Lemma \ref{lema_epsilon_m_n} we have $m<1.252n+0.9995$. Since from Lemma \ref{donje na m i n} we have $m\geq 6$ we also have $n\geq 6$ and $m<1.4n$. Using Lemma \ref{lema_donja_m_staro} yields $m>0.4999998b^{-1/2}c^{1/2}$, and finally
$$n>0.35714b^{-1/2}c^{1/2}.$$

\item Case $m$ and $n$ odd, $|z_0|=t$, $|z_1|=s$.\\
Congruences (\ref{kong3}) and (\ref{kong4})  from the proof of Lemma \ref{lema_uvodna_teorem_izbacen} hold. Let us assume to the contrary, that $n<0.5348b^{-3/4}c^{1/4}$. Then 
\begin{align*}
2t|am(m+1)-bn(n+1)|&<2tbn^2<2.00004b^{3/2}c^{1/2}n^2<0.57204c,\\
2rt(m-n)&<0.8rtn<0.800032bc^{1/2}n<0.42786c,
\end{align*}
which means that we have equality in those congruences. This implies that a contradiction can  be established as in the proof of Lemma \ref{lema_uvodna_teorem_izbacen}.\qedhere
\end{enumerate} 
\end{proof}

\begin{proof}[Proof of Theorem \ref{teorem1.5}]
Let us assume that $d>d_+$. We prove this as in Proposition \ref{propozicija_d+}
\begin{enumerate}[i)]
\item Case $b<2a$ and $c\geq 890b^4$.\\
Since $a'=4a$ and $b>10^5$, inequality 
$$308.07a'b(b-a)^2a^{-1}<1233b^3<b^4<c$$
holds. We can use Lemma \ref{lema_rickert} and Lemma \ref{lema_donja_m} and observe inequality 
$$0.5348b^{-3/4}c^{1/4}<\frac{2\log(128.08b^6c^2)\log(0.0000081b^2c^2)}{\log(bc)\log(0.00325b^{-3}c)}$$ 
and since the right-hand side is decreasing in $c$ for $c>890b^4$ we get $b<99887$ a contradiction to $b>10^5$.
\item  Case $2a\leq b\leq 12a$ and $c\geq 1613b^4$.\\
Here we  have $a'=4(b-a)$ and $\frac{b}{2}\leq b-a\leq \frac{11}{12}b$, so 
$$308.07a'b(b-a)^2a^{-1}<11400b^3<b^4<c.$$
We observe an inequality 
$$0.5348b^{-3/4}c^{1/4}<\frac{2\log(58.71b^6c^2)\log(0.052c^2)}{\log(bc)\log(0.000087903b^{-3}c)}.$$ 
After noting that the right-hand side is decreasing in $c$ for $c>1613b^4$, we deduce $b<99949$, a contradiction. 

\item Case $b>12a$ and $c\geq 52761 b^4$.\\
Let us first assume that $c\geq 52761 b^4$.
Here we  have $a'=4(b-a)$ and $\frac{b}{12}< b-a<b$, so 
$$308.07a'b(b-a)^2a^{-1}<1233b^4<c.$$ We use Lemma \ref{lema_rickert2} to obtain inequality 
$$0.5348b^{-3/4}c^{1/4}<\frac{8\log(8.40335\cdot 10^{13}b^3c)\log(0.002579c^2)}{\log(bc)\log(0.0042145b^{-4}c)}.$$ 
Similarly as before, after noting that the right-hand side is decreasing in $c$ for $c>52761b^4$, we deduce $b<99998$, a contradiction.

Now, we assume that $39247b^4<c<52761 b^4$. We can modify the method in the following way. 
For $a=1$, we only have to notice that the right-hand side in the inequality in the Lemma \ref{lema_rickert2} is decreasing in $c$, insert our lower bound on $c$, and calculate an upper bound on $b$ from the inequality. We get $b<999994$, a contradiction. For $a\geq 2$, we modify estimate $0.016858a(a')^{-1}b^{-1}(b-a)^{-2}c>0.0042145b^{-4}ca_0$, where $a_0=2$ and get $b<73454$, a contradiction. \qedhere
\end{enumerate}
\end{proof}

\begin{lemma}\label{lema_prethodi_tm16}
If $\{a,b,c,d\}$ is a $D(4)$-quadruple with $a<b<c<d_+<d$ then
$$d>\min\{0.249965b^{5.5}c^{6.5},0.240725a^{4.5}c^{5.5}\}.$$
\end{lemma}
\begin{proof}
From \cite[Lemma 5]{ireg_pro} and Theorem \ref{teorem_izbacen_slucaj} we have that $m\geq 6$ or $n\geq 7$ so Lemma \ref{lemaw4w5w6}  implies $d>0.249965b^{5.5}c^{6.5}$ or $d>0.240725a^{4.5}c^{5.5}$.
\end{proof}
\bigskip
\begin{proof}[Proof of  Theorem \ref{teorem1.6}]
Let $\{a,b,c,d\}$ be a $D(4)$-quadruple such that $a<b<c<d_+<d$ and let us assume to the contrary, that  there exists $e<c$ such that $\{e,a,b,c\}$ is an irregular $D(4)$-quadruple. Then by Lemma \ref{lema_prethodi_tm16} $c>0.240725a'^{4.5}b^{5.5}>0.240725\cdot (10^5)^{1.5}b^4>52761 b^4$, where $a'=\min\{a,e\}\geq 1$ or $c>0.249965a^{5.5}c^{6.5}>52761 b^4$. Theorem \ref{teorem1.5} implies that $\{a,b,c,d\}$ must be a regular quadruple; a contradiction.
\end{proof}

\section{Proof of Theorem \ref{teorem_prebrojavanje}} 

In this section, we aim to split our problem into several parts. We will consider separately the case when a triple $\{a,b,c\}$ is regular ($c=a+b+2r$), and when it is not regular  ($c>ab+a+b$). In the latter case, we will consider solutions of the equation $z=v_m=w_n$ without assuming that the inequalities from Lemma \ref{granice_fundamentalnih} hold when $|z_0|\neq2$. Lemmas from this section will usually also address separately the case when $c>ab+a+b$. More specifically, only this case when $(|z_0|,|z_1|)=(t,s)$ and $z_0z_1>0$, which can then be used to prove the results to all the other cases (except the case $|z_0|=2$) by using Lemma \ref{lema_nizovi}.

In what follows, we will  introduce results concerning linear forms in three logarithms. These results will establish that there are at most $3$ possible extensions of a triple to a quadruple for a fixed fundamental solution. Then, we will use a result of Laurent from \cite{laurent} in order to establish our final technical tools and finish the proof of Theorem \ref{teorem_prebrojavanje}.

First, we prove that the inequality from Lemma \ref{mnfilipin} holds also for the case $c>ab+a+b$ without assuming that the inequalities from Lemma \ref{granice_fundamentalnih} hold.

\begin{lemma}\label{mn_neregularnatrojka}
Let $\{a,b,c\}$ be a Diophantine triple with $b>\max\{a,10^5\}$ and $c>ab+a+b$. If $v_m = w_n$ has a solution with $m\geq 1$ such that $(|z_0|,x_0;|z_1|,y_1)=(t,r;s,r)$ and $z_0z_1 > 0$, then $n-1 \leq m$.
\end{lemma}  
\begin{proof}
Notice that $c>\max\{4b,ab+a+b\}$ and $c>ab+a+b>r^2$. The statement follows similarly to the argument for \cite[Lemma 2.9]{cfm}, and so we opt to omit the proof.
\end{proof}
 
\bigskip

Let $\{a,b,c\}$ be a $D(4)$ triple. We define and observe
\begin{equation}\label{lambda_linearna}
\Lambda=m\log \xi-n\log\eta+\log\mu,
\end{equation}
a linear form in three logarithms, where $\xi=\frac{s+\sqrt{ac}}{2}$, $\eta=\frac{t+\sqrt{bc}}{2}$ and $\mu=\frac{\sqrt{b}(x_0\sqrt{c}+z_0\sqrt{a})}{\sqrt{a}(y_1\sqrt{c}+z_1\sqrt{b})}$. This linear form and its variations were already studied before, for example in \cite{fil_xy4}. From \cite[Lemma 10]{fil_xy4} we know that
\begin{equation}\label{nejednakost_za_lambda}
0<\Lambda<\kappa_0\left(\frac{s+\sqrt{ac}}{2}\right)^{-2m},
\end{equation}
where $\kappa_0$ is a coefficient which is defined in the proof of the lemma with
\begin{equation}\label{kapa0}
\kappa_0=\frac{(z_0\sqrt{a}-x_0\sqrt{c})^2}{2(c-a)}.
\end{equation}
\begin{lemma}\label{lemma_kappa_linearna}
Let $(m,n)$ be a solution of the equation $z=v_m=w_n$ and assume that $m>0$ and $n>0$. Then
$$0<\Lambda<\kappa\xi^{-2(m-\delta)},$$
where
\begin{enumerate}[i)]
\item $(\kappa,\delta)=(2.7\sqrt{ac},0)$ if the inequalities from Lemma \ref{granice_fundamentalnih} hold,
\item $(\kappa,\delta)=(6,0)$ if $|z_0|=2$,
\item $(\kappa,\delta)=\left(1/(2ab),0\right)$ if $z_0=t$, 
\item $(\kappa,\delta)=\left(2.0001b/c,1\right)$ if $z_0=-t$, $b>10^5$ and $c>ab+a+b$.
\end{enumerate}
\end{lemma} 
\begin{proof}
From Lemma \ref{granice_fundamentalnih} we get 
$$0<x_0\sqrt{c}-z_0\sqrt{a}<2x_0\sqrt{c}<2.00634\sqrt[4]{ac}\sqrt{c}.$$ Inserting in the expression (\ref{kapa0}) yields $\kappa_0<2.7\sqrt{ac}$. \\
If $|z_0|=2$, equation (\ref{prva_pelova_s_a}) yields $x_0=2$. Using $c>4a$ gives us the desired estimate.\\
If $c>a+b+2r$ and $z_0=t$ then $x_0=r$ and $\kappa_0=\frac{8(c-a)}{(t\sqrt{a}+r\sqrt{c})^2}<\frac{1}{2ab}$.\\
In the last case, we observe that 
$$\kappa_0\left(\frac{s+\sqrt{ac}}{2}\right)^{-2}<\frac{2r^2c}{c-a}\cdot \frac{1}{ac}=\frac{2b\left(1+\frac{4}{ab}\right)}{c\left(1-\frac{a}{c}\right)}<2.0001\frac{b}{c}$$
where we have used that $c>ab>10^5a$.
\end{proof}

The next result is due to Matveev \cite{matveev} and can be used to get a better lower bound on the linear form (\ref{lambda_linearna}) than (\ref{nejednakost_za_lambda}).

\begin{theorem}[Matveev]\label{Matveev}
Let $\alpha_1, \alpha_2,\alpha_3$ be a positive, totally real algebraic numbers such that they are multiplicatively independent. Let $b_1,b_2,b_3$ be rational integers with $b_3\neq 0$. Consider the following linear form $\Lambda$  in the three logarithms: 
$$\Lambda=b_1\log\alpha_1+b_2\log\alpha_2+b_3\log\alpha_3.$$
Define real numbers $A_1,A_2,A_3$ by
$$A_j=\max\{D\cdot h(\alpha_j),|\log \alpha_j|\}\quad (j=1,2,3),$$
where $D=[\mathbb{Q}(\alpha_1,\alpha_2,\alpha_3):\mathbb{Q}]$. Put
$$B=\max\left\{1,\max\{(A_j/A_3)|b_j|:\ j=1,2,3\}\right\}.$$
Then we have 
$$\log|\Lambda|>-C(D)A_1A_2A_3\log(1.5eD\log(eD)\cdot B)$$
with
$$C(D)=11796480e^4D^2\log(3^{5.5}e^{20.2}D^2\log(eD)).$$
\end{theorem}
Put $\alpha_1=\xi=\frac{s+\sqrt{ac}}{2}$, $\alpha_2=\mu=\frac{\sqrt{b}(x_0\sqrt{c}+z_0\sqrt{a})}{\sqrt{a}(y_1\sqrt{c}+z_1\sqrt{b})}$ and $\alpha_3=\eta=\frac{t+\sqrt{bc}}{2}$. 
We can easily show, similarly as in \cite{sestorka} or \cite{nas2}, that 
$$\frac{A_1}{A_3}<\frac{\log(1.001\sqrt{ac})}{\log(\sqrt{bc})}<1.0001.$$
It follows by similar arguments as in \cite{fm} that 
$h(\alpha_2)<\frac{1}{4}\log(P_2)$
where
\begin{align*}
P_2=\max\{b^2(c-a)^2,x_0^2abc^2,y_1^2abc^2,x_0y_1a^{1/2}b^{3/2}c^2\}.
\end{align*}
 First, we consider the case $c=a+b+2r<4b$, so that the case $iii)$ in Lemma \ref{lema_tau} cannot hold. Also, the case where $|z_1|=s$ and $|z_0|=\frac{1}{2}(cr-st)$ cannot occur since the same lemma implies $c=a+b+2r>a^2b$ ($a=1$) and this case can be eliminated with the same arguments as in \cite{fht}. Also, since $s=a+r$  and $t=b+r$, we have $\frac{1}{2}(cr-st)=2$, so the only case is $|z_0|=|z_1|=x_0=y_1=2$. As in \cite{fm} we easily get
$$A_2=\max\{4h(\alpha_3),|\log(\alpha_3)|\}<4\log c.$$
Now we will study the case when $c>ab+a+b$ and $(|z_0|,x_0,|z_1|,y_1)=(t,r;s,r)$ without assuming the inequalities from the Lemma \ref{granice_fundamentalnih}. Here we have 
$$
x_0^2abc^2=y_1^2abc^2=(ab+4)abc^2<c^4
$$
and
$$x_0y_1a^{1/2}b^{3/2}c^2=r^2(ab)^{1/2}bc^2<c^{9/2},$$
so in this case
$$A_2<\frac{9}{2}\log c.$$

Theorem \ref{teorem1.5} implies that $c<39247b^4$, from which it follows that $A_3=2\log \eta>2\log(0.071c^{5/8})>\frac{5}{4}\log(0.014c)$. This together with Lemma \ref{mn_neregularnatrojka} implies 
\begin{align*}
B&=\max\left\{\frac{A_1}{A_3}m,\frac{A_2}{A_3},n\right\}\\
&<\max\{m,5.724,m+1\}=\max\{5.724,m+1\}.
\end{align*}
Since by Lemma \ref{donje na m i n} and Theorem \ref{teorem_izbacen_slucaj} we know that $m\geq 6$, we can use $B<m+1$. The numbers $\alpha_1$, $\alpha_2$ and $\alpha_3$ are multiplicatively independent (this can be shown similarly as in \cite[Lemma 19]{petorke}). 
We can now apply Theorem \ref{Matveev} which proves the next result.
 \begin{proposition}\label{prop_matveev}
 For $m\geq 6$ we have
 $$\frac{m}{\log (38.92(m+1))}<2.7717\cdot 10^{12}\log \eta\log c.$$
 \end{proposition}
 
\bigskip 

 If we have a $D(4)$-quadruple $\{a,b,c,d\}$ then $z=\sqrt{cd+4}$ is a solution of the equation $z=v_m=w_n$ for some $m$, $n$ and  fundamental solutions $(z_0,z_1)$. 
 
 Let $\{a,b,c\}$ be a $D(4)$-triple and let us assume that there are $3$ solutions to the equation $z=v_m=w_n$ which belong to the same fundamental solution. We denote them with $(m_i,n_i)$, $i=0,1,2$. Let us assume that $m_0<m_1<m_2$ and $m_1\geq 4$.  Denote
 $$\Lambda_i=m_i\log \xi-n_i\log\eta+\log\mu.$$
 
As in \cite{fm}, we borrow an idea of Okazaki from \cite{okazaki} in order to find a lower bound on $m_2-m_1$ in the terms of $m_0$. We omit the proof since it is analogous to \cite[Lemma 7.1]{fm}. 
\begin{lemma}\label{okazaki}
Assume that $v_{m_0}$ is positive. Then
$$m_2-m_1>\Lambda_0^{-1}\Delta \log \eta,$$
where
$$\Delta=\left|\begin{array}{cc} n_1-n_0 & n_2-n_1 \\ m_1-m_0 & m_2-m_1 \end{array} \right|>0.$$
In particular, if $m_0>0$ and $n_0>0$ then 
$$m_2-m_1>\kappa^{-1}(ac)^{m_0-\delta}\Delta \log \eta.$$
\end{lemma}

\begin{proposition}\label{prop_okazaki_rjesenja}
Suppose that there exist $3$ positive solutions $(x_{(i)},y_{(i)},z_{(i)})$, $i=0,1,2$, of the system of Pellian equations (\ref{prva_pelova_s_a}) and (\ref{druga_pelova_s_b}) with $ z_{(0)} < z_{(1)} < z_{(2)}$ belonging to the same class of solutions and $c>b>10^5$. Put $z_{(i)} = v_{m_i} = w_{n_i}$. Then 
 $m_0\leq 2$.
\end{proposition}
\begin{proof}
Let us assume to the contrary, that $m_0>2$. From Lemmas \ref{donje na m i n} and \ref{regularna_rjesenja} we know that $m_0\geq 6$ and $m_2>\kappa^{-1}(ac)^{6-\delta}\log\eta>(2.7\sqrt{ac})^{-1}(ac)^6\log\sqrt{bc}>c^5$ by Lemmas \ref{okazaki} and \ref{lemma_kappa_linearna}. After observing that the left-hand side in the inequality of Proposition \ref{prop_matveev} is increasing in $m$ we get
$$\frac{c^5}{\log(38.92(c^5+1))}<2.7717\cdot 10^{12}\log^2 c.$$
This inequality cannot hold for $c>10^5$. We conclude that $m_0\leq 2$.
\end{proof}

The next result is deduced following the logic in \cite{nas2} (with only some technical details changed), and so we omit the proof.  
\begin{proposition}\label{padeova_tm8.1}
Let $a,b,c$ be integers with $0<a<b<c$ and let $a_1=4a(c-b),$ $a_2=4b(c-a)$, $N=abz^2$, where $z$ is a solution of the system od Pellian equations (\ref{prva_pelova_s_a}) and (\ref{druga_pelova_s_b}). Put $u=c-b$, $v=c-a$ and $w=b-a$. Assume that $N\geq 10^5a_2$. Then the numbers 
$$\theta_1=\sqrt{a+a_1/N},\quad \theta_2=\sqrt{1+a_2/N}$$
satisfy
$$\max\left\{\left|\theta_1-\frac{p_1}{q} \right|,\left|\theta_2-\frac{p_2}{q}\right|\right\}>\left(\frac{512.01a_1'a_2uN}{a_1}\right)^{-1}q^{-\lambda}$$
for all integers $p_1,p_2,q$ with $q>0$, where $a_1'=\max\{a_1,a_2-a_1\}$ and
$$\lambda=1+\frac{\log\left(\frac{256a_1'a_2uN}{a_1}\right)}{\log\left(\frac{0.02636N^2}{a_1a_2(a_2-a_1)uvw}\right)}.$$
\end{proposition}

\begin{lemma}\label{lema8.3}
Let $(x_{(i)},y_{(i)},z_{(i)})$ be positive solutions to the system of Pellian equations (\ref{prva_pelova_s_a}) and (\ref{druga_pelova_s_b}) for $i\in\{1,2\}$, and let $\theta_1$, $\theta_2$ be as in Proposition \ref{padeova_tm8.1} with $z=z_{1}$. Then we have
$$\max\left\{\left|\theta_1-\frac{acy_{(1)}y_{(2)}}{abz_{(1)}z_{(2)}}\right|,\left|\theta_2-\frac{bcx_{(1)}x_{(2)}}{abz_{(1)}z_{(2)}}\right|\right\}<\frac{2c^{3/2}}{a^{3/2}}z_{(2)}^{-2}.$$
\end{lemma}
\begin{proof}
It is not hard to see that from  Proposition \ref{padeova_tm8.1} we have
\begin{eqnarray*}
\left|\sqrt{1+\frac{a_1}{N}}-\frac{p_1}{q}\right|&=\frac{y_{(1)}\sqrt{c}}{bz_{(1)}z_{(2)}}|z_{(2)}\sqrt{b}-y_{(2)}\sqrt{c}|&<\frac{4(c-b)\sqrt{c}y_{(1)}}{2b\sqrt{b}z_{(1)}z_{(2)}^2}<\frac{2c^{3/2}}{b^{3/2}}z_{(2)}^{-2}
\end{eqnarray*}
and similarly 
$$\left|\sqrt{1+\frac{a_2}{N}}-\frac{p_2}{q}\right|<\frac{2c^{3/2}}{a^{3/2}}z_{(2)}^{-2}.$$
\end{proof}

\begin{proposition}\label{prop_odnos_n1_n2}
Suppose that $\{a,b,c,d_i\}$ are $D(4)$-quadruples with $a<b<c<d_1<d_2$ and $x_{(i)},y_{(i)},z_{(i)}$ are positive integers such that $ad_i+1=x_{(i)}^2$, $bd_i+1=y_{(i)}^2$ and $cd_i=z_{(i)}^2$ for $i\in\{1,2\} $. 
 
 \begin{enumerate}[i)]
\item If $n_1\geq 8$, then $$n_2<\frac{(n_1+1.1)(3.5205n_1+4.75675)}{0.4795n_1-3.82175}-1.1.$$
More specifically,if $n_1=8$, $n_2<2628n_1$, and if $n_1\geq9$ then $n_2<83n_1$.
\item If $c>ab+a+b$ and $(z_0,z_1)=(t,s)$, $z_0z_1>0$ and $n_1\geq 9$ then $$n_2<\frac{(n_1+1)(2.5147n_1+5.11467)}{0.4853n_1-3.85292}-1<60n_1.$$
  \end{enumerate}
\end{proposition}
\begin{proof}
Put $N=abz^2$, $p_1=acy_{(1)}y_{(2)}$, $p_2=bcx_{(1)}x_{(2)}$ and $q=abz_{(1)}z_{(2)}$ in Proposition \ref{padeova_tm8.1} and Lemma \ref{lema8.3}. We get 
$$z_{(2)}^{2-\lambda}<4096a^{\lambda-3/2}b^{\lambda+3}c^{7/2}z_{(1)}^{\lambda+2}.$$
We use estimates for fundamental solutions from Lemma \ref{granice_fundamentalnih} and the inequality from the proof of Lemma \ref{lemaw4w5w6} which gives us
$$0.49999\cdot0.99999^{n_1-1}(bc)^{\frac{n_1-1}{2}-\frac{1}{4}}c<w_{n_1}<1.000011\cdot 1.00001^{n_1-1}(bc)^{\frac{n_1-1}{2}+\frac{1}{4}}c.$$
Since $z_{(1)}=w_{n_1}$, we use this inequality to show that
$$\frac{256a_1'a_2uN}{a_1}<(1.00002bc)^{n_1+3.5}$$
and
$$\frac{0.02636N^2}{a_1a_2(a_2-a_1)uvw}>(0.41bc)^{2n_1-4},$$
where we have used the assumption $n_1\geq 8$. So 
$$2-\lambda>\frac{0.4795n_1-3.82175}{n_1-2}.$$
Now we can show that 
$$z_{(2)}^{0.4795n_1-3.82175}<4096^{n_1-2}(bc)^{4.0205n_1-4}z_{(1)}^{3.5205n_1-4.17825}.$$
On the other hand 
$$z_{(1)}>1.2589^{n_1-1}(bc)^{0.49n_1-0.24}.$$
By combining these inequalities we get 
$$z_{(2)}<z_{(1)}^{\sigma},$$
where $\sigma=\frac{3.5205n_1+4.75675}{0.4795n_1-3.82175}$. If $n_2\geq n_1\sigma+1.1(\sigma-1)$, similarly as in \cite{fm}, we would get a contradiction from
$$\frac{z_{(2)}}{z_{(1)}^{\sigma}}=\left(\frac{2\sqrt{b}}{y_{(1)}\sqrt{c}+z_{(1)}\sqrt{b}}\right)^{\sigma-1}\xi^{n_2-n_1\sigma}\frac{1-A\xi^{-2n_1\sigma}}{\left(1-A\xi^{-2n_1}\right)^{\sigma}}>1,$$
where $A=\frac{y_{(1)}\sqrt{c}-z_{(1)}\sqrt{b}}{y_{(1)}\sqrt{c}+z_{(1)}\sqrt{b}}$ and $\xi=\frac{t+\sqrt{bc}}{2}$ as before. So, $n_2< n_1\sigma+1.1(\sigma-1)$ must hold, which proves the first statement. 

The second statement is proven analogously by using the inequality
$$0.7435\cdot(ab)^{-1/2}c(t-1)^{n_1-1}<w_{n_1}<1.0001(ab)^{1/2}t^{n_1-1}.$$
Notice that in this case we didn't use Lemma \ref{granice_fundamentalnih} since we know explicitly values $(z_0,z_1)=(t,s)$.
\end{proof} 

\bigskip
We now consider a linear form in two logarithms 
\begin{align*}\Gamma=\Lambda_2-\Lambda_1&=j\log\frac{s+\sqrt{ac}}{2}-k\log\frac{t+\sqrt{bc}}{2}\\
&= (m_2-m_1)\log\frac{s+\sqrt{ac}}{2}-(n_2-n_1)\log\frac{t+\sqrt{bc}}{2}\\
&=: (m_2-m_1)\log\xi-(n_2-n_1)\log\eta
\end{align*}
for which we know that $\Gamma\neq 0$ since it is not hard to show that $\xi$ and $\eta$ are multiplicatively independent. 

From Lemma \ref{lemma_kappa_linearna} we have that $0<\Lambda_i<\kappa\xi ^{-2m_1}$ for $i=1,2$, so
$$0<|\Gamma|<\kappa \xi^{-2m_1}.$$
We can now use Laurent's theorem from \cite{laurent} to find a lower bound on $|\Gamma|$, similarly as in \cite{nas2} and \cite{fm}.
\begin{theorem}[Laurent]\label{laurent}
Let $\gamma_1$ and $\gamma_2$ be multiplicatively independent algebraic numbers with $|\gamma_1|\geq 1$ and $|\gamma_2|\geq 1$. Let $b_1$ and $b_2$ be positive integers. Consider the linear form  in two logarithms 

$$\Gamma=b_2 \log \gamma_2-b_1 \log \gamma_1,$$
where $\log \gamma_1$ and $\log \gamma_2$ are any determinations of the logarithms of $\gamma_1$ and $\gamma_2$ respectively. Let $\rho$ and $\mu$ be real numbers with $\rho>1$ and $1/3\leq \mu \leq 1$. Set 
$$\sigma=\frac{1+2\mu-\mu^2}{2},\quad \lambda=\sigma \log \rho.$$
Let $a_1$ and $a_2$ be real numbers such that 
\begin{align*}
&a_i\geq \max\{1,\rho|\log\gamma_i|-\log|\gamma_i|+2Dh(\gamma_i)\},\quad i=1,2,\\
&a_1a_2\geq \lambda^2,
\end{align*}
where $D=[\mathbb{Q}(\gamma_1,\gamma_2):\mathbb{Q}]/[\mathbb{R}(\gamma_1,\gamma_2):\mathbb{R}]$. Let $h$  be a real number such that
$$h \geq \max \left\{ D\left(\log\left( \frac{b_1}{a_2'}+\frac{b_2}{a_1'}\right)+\log \lambda'+1.75\right)+0.06,\lambda',\frac{D\log 2}{2}\right\}+\log\rho.$$
Put 
$$H=\frac{h}{\lambda},\quad
\omega=2\left( 1+\sqrt{1+\frac{1}{4H^2}}\right),\quad \theta=\sqrt{1+\frac{1}{4H^2}}+\frac{1}{2H}.$$

Then
$$\log|\Lambda|\geq -C\left(h'+\frac{\lambda'}{\sigma} \right)^2a_1'a_2'-\sqrt{\omega \theta}\left(h'+\frac{\lambda'}{\sigma} \right)-\log\left(C' \left(h'+\frac{\lambda'}{\sigma} \right)^2a_1'a_2' \right)$$
with $C=C_0\mu/(\lambda^3\sigma)$ and $C'=\sqrt{C_0\omega\theta/\lambda^6}$
\begin{align*}
&C_0=\left(\frac{\omega}{6}+\frac{1}{2}\sqrt{\frac{\omega^2}{9}+\frac{8\lambda\omega^{5/4}\theta^{1/4}}{3\sqrt{a_1a_2}H^{1/2}}+\frac{4}{3}\left( \frac{1}{a_1}+\frac{1}{a_2}\right)\frac{\lambda\omega}{H}   } \right)^2.
\end{align*}
\end{theorem}
\begin{proposition}\label{prop_primjena_laurenta}
If  $z=v_{m_i}=w_{n_i}$ $(i\in \{1,2\})$ has a solution, then
$$\frac{2m_1}{\log \eta}<\frac{C_0'\mu}{\lambda^3\sigma}(\rho+3)^2h^2+\frac{2\sqrt{\omega\theta}h+2\log\left(\sqrt{C_0'\omega\theta}\lambda^{-3}(\rho+3)^2\right)+4\log h}{(\log(10^5))^2}+1,$$
where $\rho=8.2$, $\mu=0.48$
$$C_0'=\left(\frac{\omega}{6}+\frac{1}{2}\sqrt{\frac{\omega^2}{9}+\frac{16\lambda\omega^{5/4}\theta^{1/4}}{3(\rho+3)H^{1/2}\log(10^5)}+\frac{16\lambda \omega}{3(\rho+3)H\log(10^5)}}\right)^2,$$
 $h=4\log\left( \frac{2j}{\log\eta}+1 \right)+4\log\left(\frac{\lambda}{\rho+3}\right)+7.06+\log\rho$
 and $\sigma,\lambda,H,\omega,\theta$ are as in Theorem \ref{laurent}. 
\end{proposition}
\begin{proof}
Similarly as in \cite{nas2} we can take $a_i=(\rho+3)\log\gamma_i$, $i=1,2$, $h=4\log\left(\frac{2j}{\log\gamma_1}\right)+4\log\left(\frac{\lambda}{\rho+3}\right)+7.06+\log\rho$ which yields $C_0<C_0'$ as defined in the statement of the proposition. Now, we finish the proof by combining Theorem \ref{laurent} and Lemma \ref{lemma_kappa_linearna}. 
\end{proof}

\bigskip
\begin{proof}[Proof of Theorem \ref{teorem_prebrojavanje}]
Let $\{a,b,c\}$  be a $D(4)$-triple and let $N(z_0,z_1)$ denote the number of nonregular solutions of the system (\ref{prva_pelova_s_a}) and (\ref{druga_pelova_s_b}), i.e.\@ the number of integers $d>d_+$ which extend that triple to a quadruple and  which correspond to the same fundamental solution $(z_0,z_1)$. From Proposition \ref{prop_okazaki_rjesenja} we know that if we have $3$ possible solutions $m_0$, $m_1$, $m_2$ with the same fundamental solution $(z_0,z_1)$ then $m_0\leq 2$, so from Lemma \ref{donje na m i n} we know that $N(z_0,z_1)\leq 2$ for each possible pair $(z_0,z_1)$ in Theorem \ref{teorem_izbacen_slucaj}. Also, from the same theorem, we know that the case when $m$ is odd and $n$ is even cannot occur when $d>d_+$. So, if we denote with $N_{eo}$ the  number of solutions $d>d_+$ when $m$ is even and $n$ is odd, and similarly for other cases, the number of extensions of a $D(4)$-triple to a $D(4)$-quadruple with $d>d_+$ is equal to 
$$N=N_{ee}+N_{eo}+N_{oo}.$$

\medskip
\textbf{Case $iv)$}\par
This follows from Theorem \ref{teorem1.5}.

\medskip
\textbf{Case $i)$}\par 
Since $c=a+b+2r<4b$, only the case $|z_0|=|z_1|=2$ can hold as explained before Proposition \ref{prop_matveev}. This implies
$$N=N_{ee}=N\left(2,2\right) + N( -2,-2).$$
We now show that $N(2,2)\leq 1$.
Assume to the contrary, that $N(2,2)=2$. Since $\frac{1}{2}(st-cr)=2$, then $(m_0,n_0)=(2,2)$ is a solution in this case. Beside that solution, there exist two more solutions $z=v_{m_i}=w_{n_i}$, $(m_i,n_i)$, $i=1,2$, such that $2=m_0<m_1<m_2$ and $m_1\geq 8$ and $n_1\geq 8$ by Lemma \ref{donje na m i n}. From Lemmas \ref{lemma_kappa_linearna} and \ref{okazaki} we have $\kappa=6$, $\Delta\geq 4$ and $m_2-m_1>\kappa^{-1}(ac)^{m_0}\delta\log \eta$, i.e.
$$\frac{j}{\log \eta}=\frac{m_2-m_1}{\log\eta}>\frac{4}{6}(ac)^2>6.66\cdot 10^9.$$
On the other hand, $n_1\geq 8$ and Proposition \ref{prop_odnos_n1_n2} gives us
$$m_2\leq 2n_2+1\leq 2\cdot 2628n_1\leq 5256m_1+5256\leq5913m_1,$$
so $j=m_2-m_1\leq 5912m_1$.
Using Proposition \ref{prop_primjena_laurenta} yields $\frac{j}{\log\eta}<5.71\cdot 10^8$, which is a contradiction.

This proves that $N=N\left(2,2\right) + N \left( -2,-2\right)\leq 3$.

\medskip
\textbf{Cases $ii)$ and $iii)$}\par
Since $c\neq a+b+2r$, we have $c>ab+a+b$. Let $N'(z_0,z_1)$  denote a number of solutions of equation $v_m=w_n$, with $m>2$ and $b>10^5$, but without assuming that inequalities from Lemma \ref{granice_fundamentalnih} hold. Then, by Lemma \ref{lema_nizovi}  and Proposition \ref{prop_m_n_9} we have 
$$N\leq N(-2,-2)+N(2,2)+N'(z_0^{-},z_1^{-})+N'(z_0^+,z_1^+)$$
where 
$(z_0^+,z_1^+)\in \left\{ \left(\frac{1}{2}(st-cr),\frac{1}{2}(st-cr)\right), \left(t,\frac{1}{2}(st-cr)\right),\right.$ $\left. \left(\frac{1}{2}(st-cr),s\right),\right.$ $\left.\left(t,s\right)  \right\}$ and
$(z_0^{-},z_1^{-})=(-z_0^{+},-z_1^{+})$.

It is not hard to see that the previous proof for $N\left(2,2\right)\leq 1$ didn't depend for the element $c$, so it holds in this case too. We will now show that $N'(-t,-s)\leq 2$ which by Lemma \ref{lema_nizovi} implies $N'(z_0^{-},z_1^{-})\leq 2$, and $N'(t,s)\leq 2$ for $c<b^2$ and $N'(t,s)\leq 1$ for $c>b^2$, which will prove our statements in these last two cases. 

Let us assume the contrary, that $N'(-t,-s)\geq 3$ (there exist $m_0,m_1,m_2$, $2<m_0<m_1<m_2$). By Proposition \ref{prop_m_n_9} we know that $m_0\geq 9$. From Lemma \ref{lemma_kappa_linearna} we have $\Lambda<2.0001\frac{b}{c}\xi^{-2(m-1)}$ which can be used in Lemma \ref{okazaki} to get 
$$m_2>m_2-m_1>1.999c^8\log\eta>1.999c^8\log \sqrt{c}.$$
Now, we can use Proposition \ref{prop_matveev} with $B(m_2)=m_2+1$, $\log\eta<\log(c/2)$. This gives us an inequality
$$\frac{m_2}{\log (38.92(m_2+1) )}<2.81\cdot 10^{12}\log c\log (c/2).$$
The left-hand side is increasing in $m_2$ so we can solve the inequality in $c$ which yields $c<56$, a contradiction to $c>b>10^5$.

Now, let us prove that $N'(t,s)\leq 1$. Assume to the contrary that $N'(t,s)\geq 2$ and for some $3$ solutions $ 1\leq m_0<m_1<m_2$ we also have $2<m_1$ ($m_0$ is associated with a regular solution). Then by Lemma \ref{lemma_kappa_linearna}, since $\Delta\geq 1$, $c>4b$ and $b>10^5$, we get
\begin{equation}\label{jbdazadnja}\frac{m_2-m_1}{\log\eta}>2ab(\sqrt{ac})^2>8b^2>8\cdot 10^{10}.\end{equation}
On the other hand, as in the proof of \cite[Lemma 7.1]{fm}, it can be shown that $\frac{n_2-n_1}{m_2-m_1}>\frac{\log \xi}{\log \eta}$, which together with Proposition \ref{prop_odnos_n1_n2}  implies
$$\frac{m_2-m_1}{\log \eta}<\frac{n_2-n_1}{\log\xi}<\frac{f(n_1)n_1}{\log\xi},$$
where $f(n_1)=\frac{2.0294n_1+8.96759}{0.4853n_1-3.85292}\left(1+\frac{1}{n_1}\right)$. These two inequalities yield
$n_1>9.34047\cdot 10^{11}$ and $f(n_1)\leq 4.1818$.

From $0<\Lambda_1<m_1\log\xi-n_1\log\eta+\log\mu$ we have

$$\frac{m_1}{\log\eta}>\frac{m_2-m_1}{f(n_1)\log\eta}-\frac{\log\mu}{\log\eta\log\xi}>\frac{m_2-m_1}{f(n_1)\log\eta}-1.$$ So, we can use Proposition \ref{prop_primjena_laurenta} and the inequality
$$
\resizebox{\textwidth}{!}{$\frac{m_2-m_1}{f(n_1)\log\eta}<\frac{C_0'\mu}{\lambda^2\sigma}h^2(\rho+3)^2+\frac{2\sqrt{\omega\theta}h+2\log\left(\sqrt{C_0'\omega\theta}\lambda^{-3}(\rho+3)^2\right)+4\log h}{(\log(10^5))^2}+2$}
$$

 yields $\frac{m_2-m_1}{\log\eta}<152184$ which is in a contradiction to (\ref{jbdazadnja}). So we must have $N'(t,s)\leq 1$.
 
It remains to prove that $N'(-t,-s)\leq 1$ for $c>b^2$. Again, let us assume to the contrary, that there are at least $2$ solutions $m_1<m_2$  besides a solution $(m_0,n_0)=(1,1)$ (which gives $d=d_{-}(a,b,c)$). Then
$$\frac{m_2-m_1}{\log\eta}>2.0001^{-1}\frac{c}{b}(\xi)^{2(m_0-1)}\Delta>1.99\cdot 10^5.$$
After repeating steps as in the previous case, we get $f(n_1)<4.1819$ and Proposition \ref{prop_primjena_laurenta} yields $\frac{m_2-m_1}{\log\eta}<152184$, a contradiction. 

So, when $c>b^2$ we have $N\leq 2+1+2+2=7$ and when $a+b+2r\neq c<b^2$ we have $N\leq 2+1+2+1=6$. 
\end{proof}

\section{Extension of a pair}
 For completeness, to give all possible results similar to the ones in \cite{fm} and \cite{cfm}, we have also considered extensions of a pair to a triple and estimated the number of extensions to a quadruple in such cases. Extensions of a pair to a triple were considered in \cite{bf_parovi} for the $D(4)$ case. Ba\'{c}i\'{c} and Filipin have shown that a pair $\{a,b\}$ can be extended to a triple with a $c$ given by 
$$
\resizebox{\textwidth}{!}{$c=c_{\nu}^{\pm}=\frac{4}{ab}\left\{\left(\frac{\sqrt{b}\pm\sqrt{a}}{2}\right)^2\left(\frac{r+\sqrt{ab}}{2}\right)^{2\nu}+\left(\frac{\sqrt{b}\mp\sqrt{a}}{2}\right)^2\left(\frac{r-\sqrt{ab}}{2}\right)^{2\nu}-\frac{a+b}{2}\right\}$}
$$
where $\nu\in\mathbb{N}$. These extensions are derived from the fundamental solution $(t_0,s_0)=(\pm2,2)$ of the Pell equation 
$$at^2-bs^2=4(a-b),$$
associated with the problem of an extension of a pair to a triple. Under some conditions for the pair $\{a,b\}$ we can prove that these fundamental solutions are the only ones.

The next result is an improvement of \cite[Lemma 1]{bf_parovi}.
 
 \begin{lemma}\label{lema_b685a}
 Let $\{a,b,c\}$ be a $D(4)$-triple. If $a<b<6.85a$ then $c=c_{\nu}^{\pm}$ for some $\nu$.
 \end{lemma}
 \begin{proof}
 We follow the proof of \cite[Lemma 1]{bf_parovi}. Define $s'=\frac{rs-at}{2}$, $t'=\frac{rt-bs}{2}$ and $c'=\frac{(s')^2-4}{a}$. The cases $c'>b$, $c'=b$ and $c'=0$ are the same as in the \cite[Lemma 1]{bf_parovi} and yield $c=c_{\nu}^{\pm}$. It is only left to consider the case $0<c'<b$. Here we define $r'=\frac{s'r-at'}{2}$ and $b'=((r')^2-4)/a$. If $b'=0$ then it can be shown that $c'=c_{1}^-$ and $c=c_{\nu}^-$ for some $\nu$. We observe that $b'=d_{-}(a,b,c')$ so 
 $$b'<\frac{b}{ac'}<\frac{6.85}{c'}.$$
 This implies $b'c'\leq6$. If $c'=1$, since $b'>0$ and $b'c'+4$ is a square, the only possibility is $b'=5$. In that case, $a$ and $b$ extend a pair $\{1,5\}$. Then 
$$a=a_{\nu}^{\pm}=\frac{4}{5}\left\{\left(\frac{\sqrt{5}\pm1}{2}\right)^2\left(\frac{3+\sqrt{5}}{2}\right)^{2\nu}+\left(\frac{\sqrt{5}\mp1}{2}\right)^2\left(\frac{3-\sqrt{5}}{2}\right)^{2\nu}-3\right\}$$
and $b=d_+(1,5,a)=a_{\nu+1}^{\pm}$ for the same choice of $\pm$. Define $k:=\frac{b}{a}=\frac{a_{\nu+1}^{\pm}}{a_{\nu}^{\pm}}$. It can be easily shown that $k\leq 8$ and decreasing as $\nu$ increases. Also,  
$$\lim_{\nu\to\infty}\frac{a_{\nu+1}^{\pm}}{a_{\nu}^{\pm}}=\left(\frac{3+\sqrt{5}}{2}\right)^2=\frac{7+3\sqrt{5}}{2}>6.85,$$
which gives us a contradiction to the assumption $b<6.85a$. 

If $c'\in\{2,3,4,6\}$ there is no $b'$ which satisfies the necessary conditions. The case $c'=5$ gives $b'=1$ which is analogous to the previous case. 
  \end{proof}
  \begin{rem}
  For a pair $ (a,b)=(4620,31680)$, where $b>6.85a$, we have a solution $(s_0,t_0)=(68,178)$ of the  Pellian equation (\ref{par_trojka}), so it can be extended with a greater element to a triple with $c\neq c_{\nu}^{\pm}$. For example $c=146 434 197$. So, this result cannot be improved further more.
  \end{rem}
  
  \begin{proof}[Proof of Proposition \ref{prop_covi}]
  We have
  \begin{align*}
  c_1^{\pm}&=a+b\pm 2r,\\
  c_2^{\pm}&=(ab+4)(a+b\pm2r)\mp4r,\\
  c_3^{\pm}&=(a^2b^2+6ab+9)(a+b\pm2r)\mp4r(ab+3),\\
  c_4^{\pm}&=(a^3b^3+8a^2b^2+20ab+16)(a+b\pm2r)\mp4r(a^2b^2+5ab+6),\\
  c_5^{\pm}&=(a^4b^4+10a^3b^3+35a^2b^2+50ab+25)\mp4r(a^3b^3+7a^2b^2+15ab+10).
  \end{align*}
The aim is to use Theorem \ref{teorem_prebrojavanje} and since $N=0$ for $b<10^5$ we can use the lower bound on $b$ (more precisely, on $b$ if $b<c$ and on $c$ otherwise).   

Case $c_1^{\pm}$ implies that $\{a,b,c\}$ is a regular triple so $N\leq 3$. 

If $a=1$ then $c_2^-<b^2$ so the best conclusion is $N\leq 7$. 
On the other hand, if $c=c_2^-$ and $a^2\geq b$ we have $c>b^2$ since $a+b-2r\geq 1$ so $N=0$. It remains to consider the case when $b>a^2$. Then it can be shown that $r^2<0.004b^2$ and $c_2^->0.872ab^2$, so if $a\geq2$ we can again conclude $N\leq 6$.
Also if $c\geq c_2^+=(a+b)(ab+4)+2r(ab+2)>b^2$ we have $N\leq 6$.
Observe that $c\geq c_5^{-}>a^4b^4$. If $b\leq7104a$ then $c>\frac{10^{20}}{7104^4}b^4>39263b^4$. It follows that $N=0$ by Theorem \ref{teorem_prebrojavanje}. If $b>7104a$ then $a+b-2r>0.97b$ so $c_5^{-}>0.97ba^4b^5>97000b^4$, and again $N=0$.

Similarly, we observe $c_4^->a^3b^3$, and if $b\leq63a$ we get $N=0$. If $b>63a$ we have $c+b-2r>0.76b$ which will lead to $N=0$ when $a\geq 38$. Cases $a\leq 37$ can be studied separately. We remark that only $a\geq 35$ led to $N=0$, and others to $N\leq 6$. 
  \end{proof}

  From Proposition \ref{prop_covi} and Lemma \ref{lema_b685a}  we can conclude the result of Corollary \ref{kor2} after observing that $10^5<b<6.85a$ implies $a\geq 14599$.

\bigskip
\textbf{Acknowledgement:} The author was supported by the Croatian Science Foundation under the project no.\@ IP-2018-01-1313.\\


\end{document}